\documentclass[12pt]{amsart}

\usepackage[hmargin=0.8in,height=8.6in]{geometry}
\usepackage{amssymb,amsthm, times}
\usepackage{delarray,verbatim, soul}
\setstcolor{red}

\usepackage{ifpdf}
\ifpdf
\usepackage[pdftex]{graphicx}
 \else
\usepackage[dvips]{graphicx}
 \fi

\usepackage{bm}

\newcommand {\ud}{{\rm d}}

\newcommand {\B}{\mathcal{B}}

\usepackage{ifpdf}
\usepackage{color}

\usepackage{xcolor}
{\begin{mdframed}[backgroundcolor=yellow]\begin{remark}}%
		{\end{remark}\end{mdframed}}
\definecolor{webgreen}{rgb}{0,.5,0}
\definecolor{webbrown}{rgb}{.8,0,0}
\definecolor{emphcolor}{rgb}{0.95,0.95,0.95}
\usepackage{hyperref}
\hypersetup{%
	colorlinks=true,
	linkcolor=webbrown,
	filecolor=webbrown,
	citecolor=webgreen,
	breaklinks=true}
\ifpdf \hypersetup{pdftex,
	pdfstartview=FitH, 
	bookmarksopen=true,
	bookmarksnumbered=true
} \else \hypersetup{dvips} \fi

\numberwithin{equation}{section}

\newtheorem{theorem}{Theorem}[section]
\newtheorem{proposition}{Proposition}[section]

\newtheorem{remark}{Remark}[section]
\newtheorem{lemma}{Lemma}[section]

\newtheorem{assump}{Assumption}[section]

\numberwithin{remark}{section} \numberwithin{proposition}{section}
\numberwithin{corollary}{section}
\newcommand {\gen}{\mathcal{L}}
\newcommand {\R}{\mathbb{R}}

\newcommand {\F}{\mathcal{F}}
\newcommand {\A}{\mathcal{A}}
\newcommand {\N}{\mathbb{N}}
\newcommand {\p}{\mathbb{P}}

\newcommand {\E}{\mathbb{E}}

\newcommand{\diff}{{\rm d}}

\newcommand{\lev}{L\'{e}vy }

\newcommand{\red}{\textcolor[rgb]{1.00,0.00,0.00}}

\newcommand{\mazenta}{\textcolor[rgb]{0.65,0.00,0.35}}

\newcommand{\e}{\mathbb{E}}

\newcommand{\eqdef}{\raisebox{0.4pt}{\ensuremath{:}}\hspace*{-1mm}=}

\newcommand{\bF}{\mathbf{F}}
\newcommand{\Prob}{\mathbb{P}}
\newcommand{\Probm}{\mathbf{P}}
\newcommand{\Em}{\mathbf{E}}

\newcommand{\DD}{\mathcal{D}}
\newcommand{\bb}{\mathbf{b}}
\title[Bailout periodic dividend problem for Markov additive processes]{On the bailout dividend problem with periodic dividend payments for spectrally negative Markov additive processes}
\author[D. Mata L\'opez]{Dante Mata}
\address[Dante Mata]{Department of Probability and Statistics, Centro de Investigaci\'{o}n en Matem\'{a}ticas A.C. Calle Jalisco s/n. C.P. 36240, Guanajuato, Mexico}
\email{dante.mata@cimat.mx}
\author[H. A. Moreno-Franco]{Harold A. Moreno-Franco}
\address[Harold A. Moreno-Franco]{Department of Statistics and Data Analysis, HSE University, Moscow, Russian Federation.}
\email{hmoreno@hse.ru}
\author[K. Noba]{Kei Noba}
\address[K. Noba]{School of Statistical Thinking, The Institute of Statistical Mathematics
10-3 Midori-cho, Tachikawa Tokyo 190-8562, Japan.}
\email{knoba@ism.ac.jp}
\author[J. L. P\'erez]{Jos\'e-Luis P\'erez}
\address[J. L. P\'erez]{Department of Probability and Statistics, Centro de Investigaci\'{o}n en Matem\'{a}ticas A.C. Calle Jalisco s/n. C.P. 36240, Guanajuato, Mexico}
\email{jluis.garmendia@cimat.mx}

\begin{document}
		\begin{abstract}
		This paper studies the bailout optimal dividend problem with regime switching under the constraint that dividend payments can be made only at the arrival times of an independent Poisson process while capital can be injected continuously in time. We show the optimality of the regime-modulated Parisian-classical reflection strategy when the underlying risk model follows a general spectrally negative Markov additive process. In order to verify the optimality, first we study an auxiliary problem driven by a single spectrally negative \lev process with a final payoff at an exponential terminal time and characterise the optimal dividend strategy. Then, we  use the dynamic programming principle to transform the global regime-switching problem into an equivalent local optimization problem with a final payoff up to the first regime switching time. The optimality of the regime modulated Parisian-classical barrier strategy can be proven by using the results from the auxiliary problem and approximations via recursive iterations.
\ \\
\ \\
\noindent \small{\textbf{Keywords:}
	regime switching; spectrally one-sided \lev processes; scale functions; periodic and singular control strategies.\\
\ \\
\noindent  \textbf{Mathematics Subject Classification}: 60G51, 93E20, 91G80}\\
	\end{abstract}
	\maketitle
	\section{Introduction} 
	In the bailout model of de Finetti's dividend problem, the goal is to find a joint optimal dividend and capital injection strategy in order to maximise the expected net present value (NPV) of dividend payments minus the capital injections. A spectrally negative L\'evy process, namely a L\'evy process with no positive jumps, has been used to model the surplus for an insurance company that has a diffusive behavior because of the premiums and jumps downwards by claim payments. In the seminal paper \cite{AvrPalPis}, Avram et  al. showed that it is optimal to inject capital by reflecting the surplus process from below at zero and pay dividends from above at a suitable chosen threshold.
	
	
 	As in \cite{AvrPalPis}, most of the existing continuous-time models assume that the dividends can be paid at all times and instantaneously (see, e.g., \cite{AvrPalPis,AM2015,BKY,JP2012}); but in reality, dividend-payout decisions can only be made at discrete times, for that reason, the modeling of optimal dividend-payout in discrete random times has recently drawn much attention; see, e.g., \cite{ABT,ACT,ATW,NobPerYamYan,NobPerYu,PerYam}. 
	
	With this in mind, in this work, we impose the constraint that dividend payments can only be made at discrete times given by the arrival times of a Poisson process, independent of the surplus process. In addition, the classical bailout restriction requires that capital must be injected continuously in time so that the controlled process remains non-negative uniformly in time. 
	
	In this paper, we consider the bailout dividend problem in a more general framework, where the underlying surplus is driven by a spectrally negative Markov additive process. This process can be seen as a family of spectrally negative \lev processes switching via an independent Markov chain. The regime-switching model is often used to capture the changes in market behavior due to macroeconomic transitions or macroeconomic readjustments, such as technological development, epidemics, and geopolitical issues. The continuous-time Markov chain is commonly used to approximate some stochastic factors that affect the underlying state processes. In addition, a negative jump is introduced each time there is a change in the current regime. This jump is independent of the family of \lev processes and the Markov chain and can be interpreted as the cost for the company to adapt to the new regime. The regime-switching model turns out to be attractive in financial applications as it provides tractable and explicit structures, and it has become a vibrant research topic in the past decades. Some recent work motivated by different financial applications can be found in \cite{BL2009,HSZ,KN2020,CourtoisSu2020}.
	
	Optimal dividend problems in the context with regime-switching have been studied mostly in the framework of jump diffusion models, 
	 see e.g. \cite{AM2015,J2015,J2019,JP2012}; however, recently Noba et al. \cite{NobPerYu} studied the more general spectrally negative \lev framework. Similar to the single regime work, these works have shown that optimal dividend strategies fit in the type of barrier strategies as well. These previous studies, assume that dividends can be paid continuously in time, so it becomes an open question whether a barrier dividend policy is optimal in the regime switching case under the constraint that dividend payments can only be made at the jump times of an independent Poisson process.
	
	This paper aims to provide a positive answer to the optimality of the periodic-classical barrier dividend strategy, namely the periodic dividend payment and classical capital injection but modulated by the regime states. Our approach on showing the optimality of barrier strategies relies on purely probabilistic methods and is based on fluctuation identities for spectrally negative \lev processes reflected at Poissonian times. The motivation behind periodic-classical barrier strategies arises from the works by Noba et al. \cite{NobPerYamYan} and P\'erez and Yamazaki \cite{PerYam}, where optimality was shown in the single regime context for spectrally negative and spectrally positive \lev processes, respectively. However, our analysis differs from \cite{NobPerYamYan} due to the complexity caused by different regimes. The verification of optimality of barrier strategies is expected to be much more involved than in \cite{NobPerYamYan} as the barrier in each regime is coupled with other regime modulated barriers through the definition of the value function. The HJB variational inequalities that arise for the global control problem become a system of coupled variational inequalities based on the regime states. In order to reduce complexity and deal with the switch in regimes, we borrow the idea of stochastic control to use the dynamic programming principle and localize the problem up to the period of the first regime switch, see \cite{NobPerYu} and \cite{YangZhu} for similar optimal dividend problems.
	
	Our verification of optimality can be summarised as follows:
	\begin{enumerate}
		\item First, we study an auxiliary bailout dividend problem with a terminal payoff until an independent exponential time driven by a single spectrally negative \lev process. In this part, we compute the expected NPV of dividends minus capital injections under a periodic-classical reflection strategy explicitly in terms of the scale function, and perform the ``guess and verify'' procedure common in the literature. It is noteworthy that we present a novel result, where we compute the resolvent density for the spectrally negative \lev process with periodic reflection above up to the first downcrossing time below 0, as well as the resolvent density for the spectrally negative \lev process with periodic reflection above and classical reflection below. The candidate optimal barrier is chosen using the conjecture that the slope of the value function at the barrier becomes one, 
		then we proceed to verify the optimality of the selected barrier strategy by showing that the candidate value function solves the proper variational inequalities.
		\item After studying the single spectrally negative \lev model, we define an iteration operator by proving the dynamic programming principle similar to \cite{NobPerYu} and \cite{YangZhu}. We can show the existence of the candidate optimal barriers modulated by the regime states using the results from step (1). Then we proceed to prove that both the value function and the expected NPV under the regime-modulated periodic-classical reflection strategy are solutions of a functional equation, and then we prove via iterative methods that the expected NPV of the candidate barrier strategy agrees with the value function. This completes the second step of the verification and the optimality of the barrier type control is successfully retained in the general framework as conjectured.
	\end{enumerate}

	The rest of this paper is structured as follows. Section \ref{sec_levy_def} introduces some mathematical preliminaries regarding spectrally negative \lev processes. In Section \ref{prob_map} we formulate the bailout dividend problem with regime switching in the spectrally negative Markov additive model, with periodic dividend decision times. The main result in this section confirms the optimality of the regime modulated periodic-classical reflection strategy. Section \ref{strategy} then formulates the auxiliary bailout dividend problem with Poissonian dividend decision times and with a final payoff at an independent exponential time. Our main result in this section gives the optimality of the periodic-classical reflection strategy. In Section \ref{sec_NPV_Scale} we compute the expected NPV of the periodic-classical reflection strategy in terms of the scale function; in addition, we present new results on fluctuation theory for spectrally negative \lev processes, namely the computation of resolvents for the process with periodic reflection above and for the process with periodic reflection above and classical reflection below. Sections \ref{sec_candidate_optimal} and \ref{verification_aux} give the construction and existence of the candidate optimal strategy and the rigorous verification of the optimality for the auxiliary problem, respectively. Finally, in Section \ref{sec_optimal_MAPS} we define an auxiliary iteration operator and provide the verification of optimality of the regime modulated barrier strategy via iterative arguments.  {Throughout the paper, the right hand derivative of a real function $f$ is denoted by $f'_+(x)$, whenever it exists.}
	\section{Preliminaries on spectrally negative \lev processes}\label{sec_levy_def} 
	Let us consider a spectrally negative L\'evy process $X=(X(t); t\geq 0)$  defined on a  probability space $(\Omega, \mathbf{F}, \p)$ {where $\mathbf{F}:=\{\mathcal{F}_t:t\geq0 \}$ denotes the right-continuous filtration generated by $X$}.  For $x\in \R$, we denote by $\p_x$ the law of $X$ when it starts at $x$ and write for convenience  $\p$ in place of $\p_0$. Accordingly, we shall write $\e_x$ and $\e$ for the associated expectation operators. 
	
	We denote by $\psi_{X}:[0,\infty)\rightarrow\R$  to the Laplace exponent of the process $X$, i.e. 
		\[
		\e\big[{\rm e}^{\theta X(t)}\big]=:{\rm e}^{\psi_X(\theta)t}, \qquad t, \theta\ge 0,
		\]
		given by the \emph{L\'evy-Khintchine formula} 
		\begin{equation}\label{lk}
			\psi_X(\theta) :=\gamma\theta+\dfrac{\eta^2}{2}\theta^2+\int_{(-\infty,0)}\big({\rm e}^{\theta z}-1-\theta z \mathbf{1}_{\{z >-1\}}\big)\Pi(\ud z), \quad \theta \geq 0.
		\end{equation}
		Here, $\gamma\in \R$, $\eta \ge 0$, and $\Pi$ is the \lev measure of $X$ defined on $(-\infty,0)$ which satisfies
		\begin{equation*}
		\int_{(-\infty,0)}(1\land z^2)\Pi(\ud z)<\infty.
		\end{equation*}	
		It is well known that  $X$ has paths of bounded variation if and only if $\eta=0$ and $\int_{(-1, 0)} |z|\Pi(\mathrm{d}z)$ is finite. In this case $X$ can be written as
		\begin{equation}
			X(t)=ct-S(t), \,\,\qquad t\geq 0,\notag
		\end{equation}
		where 
		\begin{align*}
			c:=\gamma-\int_{(-1,0)} z\Pi(\mathrm{d}z) 
		\end{align*}
		and $(S(t); t\geq0)$ is a driftless subordinator. {We assume that the process $X$ does not have monotone paths, and therefore we must have  $c>0$ and we can write}
		\begin{equation*}
			\psi_X(\theta) = c \theta+\int_{(-\infty,0)}\big( {\rm e}^{\theta z}-1\big)\Pi(\ud z), \quad \theta \geq 0. 
		\end{equation*}
	\subsection{Scale functions}\label{App_Scale}
	For fixed $q \geq 0$, let $W^{(q)}: \R \to [0 , \infty)$ be the scale function of the spectrally negative 
	L\'evy process $X$. 
	This takes the value zero on the 
	negative half-line, and on the positive half-line it is a continuous and strictly 
	increasing function  defined by its Laplace transform: 
	\begin{align}\label{scale_function_laplace}
		\int_0^\infty e^{-\theta x} W^{(q)} (x) \diff x 
		= \dfrac{1}{\psi_X (\theta) -q} , ~~\theta > \Phi (q), 
	\end{align}
	where $\psi_X$ is as in \eqref{lk} and
	\begin{align}
		\begin{split}
			\Phi(q) := \sup \{ \lambda \geq 0: \psi_X(\lambda) = q\} . 
		\end{split}
		\label{def_varphi}
	\end{align}
	We also define, for all $x \in \R$, 
	\begin{align*}
		&{\overline{W}}^{(q)} (x) := \int_0^x W^{(q)} (y) \diff y,  ~~~
		{\overline{\overline{W}}}^{(q)} (x) : = \int_0^x \int_0^z W^{(q)} (y)\diff y \diff z, \\
		&Z^{(q)} (x) := 1+ q\overline{W}^{(q)}(x), ~~~
		\overline{Z}^{(q)} (x) := \int_0^x Z^{(q)} (z) \diff z=x +q\overline{\overline{W}}^{(q)} (x) . 
	\end{align*}
	Because $W^{(q)}(x) = 0$ for $-\infty < x < 0$, we have
	\begin{align*}
		\overline{W}^{(q)}(x) = 0,\quad \overline{\overline{W}}^{(q)}(x) = 0,\quad Z^{(q)}(x) = 1,
		\quad \textrm{and} \quad \overline{Z}^{(q)}(x) = x, \quad x \leq 0.  
	\end{align*}
	
	\begin{remark} \label{remark_scale_function_properties}
		\begin{enumerate}
			\item[(1)] $W^{(q)}$ is differentiable a.e.. 
			In particular, if $X$ is of unbounded variation or the \lev measure is atomless, it is known that $W^{(q)}$ is $\mathcal{C}^1(\R \backslash \{0\})$; see, e.g.,\ \cite[Theorem 3]{Chan2011}.
			\item[(2)] 
			As in Lemma 3.1 of \cite{KKR},
			\begin{align*} 
			\begin{split}
			{W^{(q)}} (0) &= \left\{ \begin{array}{ll} 0 & \textrm{if $X$ is of unbounded
				variation,} \\ \dfrac 1 {c} & \textrm{if $X$ is of bounded variation,}
			\end{array} \right. 
			\end{split}
			\end{align*}
		\end{enumerate}
	\end{remark}


From the identity (6) in \cite{lrz}, 
\begin{align}\label{eq1}
\begin{split}
W^{(q+r)}(x)-W^{(q)}(x)&=r\int_0^xW^{(q+r)}(u)W^{(q)}(x-u) \diff u,\\
Z^{(q+r)}(x)-Z^{(q)}(x)&=r\int_0^xW^{(q+r)}(u)Z^{(q)}(x-u) \diff u,  
\end{split}
\qquad x \in \R.
\end{align}

We also define, for $q, r \in(0,\infty)$ and $x \in \R$,
\begin{align}\label{def_z_nuevo}
Z^{(q)}(x,\Phi(q+r)) &:=e^{\Phi(q+r) x} \left( 1 -r \int_0^{x} e^{-\Phi(q+r) z} W^{(q)}(z) \diff z	\right)  \\
&=r \int_0^{\infty} e^{-\Phi(q+r) z} W^{(q)}(z+x) \diff z	 > 0,\notag 
\end{align}
where the second equality holds due to \eqref{scale_function_laplace}.\\
By differentiating \eqref{def_z_nuevo} with respect to the first argument,
\begin{align}\label{eq5}
Z^{(q) \prime}(x,\Phi(q+r)) &:= \dfrac \partial {\partial x}Z^{(q)}(x,\Phi(q+r))  = \Phi(q+r) Z^{(q)}(x,\Phi(q+r))	- r W^{(q)}(x), \quad x > 0. 
\end{align}

Finally, for $b\geq 0$ and $x \in \R$,  we define
\begin{align}
\begin{split}
	W_{b}^{(q, r)} (x ) 
	&:= W^{(q)}(x) + r \int_b^{x } W^{(q + r)} (x-y)W^{(q)} (y) \diff y,\\
	Z_{b}^{(q, r)} (x ) 
	&:= Z^{(q)}(x) + r \int_b^{x } W^{(q + r)} (x-y) Z^{(q)} (y) \diff y, 
	\\
	\overline{Z}_{b}^{(q, r)} (x ) 
	&:= \overline{Z}^{(q)}(x) + r \int_b^{x } W^{(q + r)} (x-y) \overline{Z}^{(q)} (y) \diff y. 
\end{split}
\label{identities} 
\end{align}
Notice that the identities in \eqref{identities} reduce to
\begin{equation}\label{Rem_identities_simplify}
W_{b}^{(q, r)} (x ) =  W^{(q)}(x),\quad Z_{b}^{(q, r)} (x ) = Z^{(q)}(x), \quad \overline{Z}_{b}^{(q, r)} (x ) = \overline{Z}^{(q)}(x),\quad\text{when}\ x \in [0,b]. 
\end{equation}
 In addition, for $x \in \R$ we have
\begin{equation}\label{Rem_identities_simplify.0}
W_{0}^{(q, r)} (x ) = W^{(q+r)}(x), \quad Z_{0}^{(q, r)} (x ) = Z^{(q+r)}(x).
\end{equation}
For a comprehensive study on the scale functions and their application, see \cite{KKR,Kyp}.\\
 {Finally}, let us introduce the following notation that will be used throughout this paper. For  {$x,b\in[0,\infty)$, and a measurable function $h:\R\mapsto\R$, we define}
\begin{align}
\rho_b^{(q)}(x;h)&:=\int_0^{b}W^{(q)}(x-y)h(y)dy\label{fun_rho},\\
\rho_{b}^{(q, r)} (x ;h) &:= \rho^{(q)}_{b}(x;h) + r \int_b^{x } W^{(q + r)} (x-y) \rho^{(q)}_{b} (y;h) \diff y,\label{fun_rho_0},\\
\Xi^{(q,r)}(b;h)& :=\displaystyle \int_0^{\infty}h(y+b)e^{-\Phi(q+r)y}\diff y+r\int_b^{\infty}e^{-\Phi(q+r)(y-b)}\rho^{(q)}_{b}(y;h)\diff y.\label{e4}
\end{align}
 
\section{{The bailout optimal dividend problem with periodic dividend payments and regime switching}}\label{prob_map}

	We formulate the dividend problem  {when the surplus is driven by a Markov additive process (MAP) with negative jumps}, and present our main result that states the optimality of barrier strategies.

	\subsection{Spectrally negative Markov additive processes}
	Let us consider a bivariate process $(X,H) = \lbrace (X(t),H(t)); t\geq 0 \rbrace$, where the component $H$ is a continuous-time Markov chain with finite state space $E = \lbrace 1,\cdots, N \rbrace$ and generator matrix $Q=(\lambda_{ij})_{i,j\in E}$. When the chain $H$ is in state $i$, $X$ behaves as a \textit{spectrally negative} L\'evy process $X^i$. In addition, when then process $H$ changes to a state $j \neq i$, the process $X$ jumps according to a non-positive random variable $J_{ij}$ with $i,j \in E$. The components $(X^i)_{i \in \E}, H, \text{ and } (J_{ij})_{i,j \in E}$ are assumed to be independent and are defined on some filtered probability space $(\Omega, \F, \bF, \Prob)$, where $\bF:=(\F_t)_{t \geq 0}$ is the right-continuous complete filtration generated by the processes $(X,H)$ and the family of random variables $(J_{ij})_{i,j \in E}$. We denote by $\Probm_{(x,i)}$ the law of the process conditioned on the event $\lbrace X(0) = x, H(0) = i \rbrace$; likewise we denote by $\Em_{(x,i)}$ the associated expectation operator.  
	
	Throughout this work we assume that for each $i \in E$, the Laplace exponent of the \lev process $X^i$, $\psi_{X^i}: [0,\infty) \rightarrow \R$, is given by the \textit{\lev-Khintchine} formula
	\[
	\psi_{X^i}(\theta) =\gamma_i \theta+\dfrac{\eta_i^2}{2}\theta^2+\int_{(-\infty,0)}\big({\rm e}^{\theta z}-1-\theta z \mathbf{1}_{\{z >-1\}}\big)\Pi(i,\ud z), \quad \theta \geq 0,
	\]
	where $\gamma_i \in \R, \, \eta_i \geq 0$, and $\Pi(i,\cdot)$ is the \lev measure of $X^i$ on $(-\infty,0)$ that satisfies $\int_{(-\infty,0)} (1 \wedge x^2) \Pi(i, \diff x) < \infty $. In addition, as in Section \ref{sec_levy_def}, if $X^i$ has paths of bounded variation its Laplace exponent is given by $\psi_{X^i}(\theta) = c_i \theta + \int_{(-\infty,0)} (e^{\theta z} - 1) \Pi(i, \diff z), \, \theta \geq 0 $, where $c_i := \gamma_i - \int_{(-1,0)} z \Pi(i, \diff z)$.
	
	 Throughout this work we denote by $\Prob_x^i$ the law of the \lev process $X^i$ contidioned on the event $\lbrace X^i(0) = x \rbrace$ and by $\E_x^i$ its associated expectation operator.
	\subsection{Bailout optimal dividend problem with Poissonian decision times and regime switching.}\label{b1}
	A strategy is a pair of non-decreasing, right-continuous, and adapted processes $\pi := (L_r^{\pi}(t),R_r^{\pi}(t))$ consisting of the cumulative amount of dividends $L_r^{\pi}$ and those of capital injection $R_r^{\pi}$. 
	
	Throughout  this paper we will consider  that the dividend payments can only be made at the arrival times $\mathcal{T}_r :=(T(i); i\geq 1 )$ of a Poisson process $N^r=( N^r(t); t\geq 0) $ with intensity $r>0$, which is defined on  $(\Omega, \F, \mathbb{F}, \Prob)$, where $\mathbb{F}=\{\mathcal{H}_{t}\}_{t\geq0}$ is the right-continuous complete filtration generated by $(X,N^r)$. We assume that $X$ and $N^r$ are independent on the previous probability space. In other words, we consider that $L_r^{\pi}$ admits the form
	 \begin{align}
		L_r^{\pi}(t)=\int_{[0,t]}\nu^{\pi}(s)\diff N^r(s),\qquad\text{$t\geq0$,} \label{restriction_poisson}
	\end{align}
	for some c\`agl\`ad process $\nu^{\pi}$ adapted  {to the filtration generated by $X$, $H$ and $N^r$}. 

	 The process $R_r^{\pi}$ is   non-decreasing, right-continuous, and $\bF$-adapted, with $R_r^{\pi}(0-) = 0$. Contrary to the dividend payments, capital injection can be made continuously in time. In addition, the process $R_r^{\pi}$ must satisfy
	\begin{equation}\label{bailout_admissible}
	\Em_{(x,i)}\left[ \int_{[0,\infty)} e^{-I(t)} \diff R_r^{\pi}(t) \right] < \infty, \quad x \geq 0, i \in E,
	\end{equation}
	where $I(t):= \int_0^t q(H(s)) \diff s,$ and $q: E \rightarrow \R_+$ represents the Markov-modulated rate of discounting.
	The corresponding  {controlled} process associated to the strategy $\pi$ is given by $U_r^{\pi}(0-) = X(0)$ and 
	\begin{align*}
		U_r^{\pi}(t) := X(t) - L_r^{\pi}(t) + R_r^{\pi}(t), \quad t \geq 0.
	\end{align*}
	 We denote by $\A$ the set of strategies satisfying the constraints mentioned above and  {that} $U_r^{\pi}(t) \geq 0$ for all $t \geq 0$ a.s.. We call a strategy $\pi$ \textit{admissible} if $\pi \in \A$.
	 
	We consider that $\beta > 1$ is the constant cost per unit of capital injected in all regimes. Our aim is to maximize the expected net present value (NPV)
	\begin{equation}\label{va1}
	V_{\pi}(x,i) := \Em_{(x,i)} \left[ \int_{[0, \infty)} e^{-I(t)} \diff L_r^{\pi}(t) - \beta \int_{[0, \infty)} e^{-I(t)} \diff R_r^{\pi}(t) \right], \quad x \geq 0, i \in E, 
	\end{equation}
	over all $\pi \in \A$. Hence, our goal is to find the value function of the problem
	\begin{equation}\label{MAP_Value}
		V(x,i) := \sup_{\pi \in \A} V_{\pi}(x,i), \quad x \geq 0,
	\end{equation}
and obtain an optimal strategy $\pi^* \in \A$   whose expected NPV, $V_{\pi^{*}}$, agrees with $V$ if such a strategy exists.

Throughout this paper we assume the following.

\begin{assump}\label{assump_1} We assume that $\E[X^i(1)] =   {\psi_{X^i}'(0+)} > -\infty$ for $i\in E$.
\end{assump}
\begin{assump}\label{assump_2} For all $i,j \in E$ with $i \neq j$, we assume that $\max_{i,j \in E} \E[|J_{ij}|] < \infty$.
\end{assump}
We claim that the dynamic programming principle for the value function of the control problem holds valid, which will play a key role in the verification via iteration operators later on (see Section \ref{sec_optimal_MAPS}). We defer its  proof to the Appendix (see Subsection \ref{App_DPP_Proof}). 
\begin{proposition} 
	\label{dpp}  
	For $x \in \R$ and $i \in E$, we have 
	\begin{equation}\label{MAP_DPP}
		V(x,i) = \sup_{\pi \in \A} \Em_{(x,i)} \left[ \int_{[0, \zeta)} e^{-I(t)} \diff L_r^{\pi}(t) - \beta \int_{[0, \zeta)} e^{-I(t)} \diff R_r^{\pi}(t) + e^{-I(\zeta)} V\mazenta{\hat{V}?}(U_r^{\pi}(\zeta), H(\zeta)) \right],
	\end{equation}
 	 {where $\zeta$ denotes the epoch of the first regime switch.}
\end{proposition}
\subsection{Markov-modulated periodic-classical barrier strategies}
For our candidate optimal control, we will consider the Markov-modulated reflection strategy, say $\pi^{0,\bb} = (L_r^{0,\bb}(t),R_r^{0,\bb}(t); t \geq 0)$, at a suitable reflection threshold $\bb = (b(i))_{i \in E}$. Namely, dividends are paid as a lump sum whenever the surplus process is above at $b (H(T(i)))$, where $T(i) \in \mathcal{T}_r$ is the $i$-th arrival time of the Poisson process $N^r$, while it is pushed upward by capital injection whenever it attempts to down cross zero. The resulting surplus process becomes the spectrally negative MAP with periodic and classical reflection, denoted by $ U_r^{0,\bb}(t):= X(t) - L_r^{0,\bb}(t) + R_r^{0,\bb}(t), t\geq 0$. We can describe explicitly the cumulative dividend payments associated to the Markov-modulated barrier strategy as
\[
L_r^{0,\bb}(t) = \sum_{\substack{T(i) \in \mathcal{T}_r \\ T(i)  \leq t} } (U_r^{0,\bb}(T(i)-) - b(H(T(i)))) \vee 0, \quad t \geq 0.
\]
 {By a modification of Remark 3.5 in \cite{NobPerYu}, it follows that the Markov-modulated barrier strategy $\pi^{0,\bb}$ is indeed admissible.}

We state the main result of our paper, and its proof will be provided by an iterative construction of the value function $V$ in Section \ref{sec_optimal_MAPS}.
\begin{theorem}\label{thm_optimal_barrier}
Under Assumptions \ref{assump_1} and \ref{assump_2}, there exists $\bb^* = (b^*(i))_{i\in E}$ such that the Markov-modulated reflection strategy with Poissonian decision times, $\pi^{0,\bb^*}$, is optimal and the value function of the problem \eqref{MAP_Value} is given by
\[
V(x,i) = V_{\pi^{0,\bb^*}}(x,i), \quad \text{ for } x\geq 0, i \in E.
\]
\end{theorem}
	\section{
		{Optimal strategies for an auxiliary Poissonian bail-out dividend problem with an exponential terminal time}}\label{strategy} 
	
	{In this section we introduce a Poissonian bail-out dividend problem with an exponential terminal time, in a model with a single spectrally negative L\'evy process, which is closely related with the problem mentioned in Section \ref{prob_map}, due to Proposition \ref{dpp}. To introduce the problem, let us first assume that the uncontrolled process is given by a spectrally negative \lev process $X$ with Laplace exponent, denoted by $\psi_X$, as in Section \ref{sec_levy_def}.} 
	
	{As in Section \ref{b1}, we consider that the dividend payments can only be made at the arrival times $\mathcal{T}_r :=(T(i); i\geq 1 )$ of a Poisson process $N^r=( N^r(t); t\geq 0) $ with intensity $r>0$, which is defined on  $(\Omega, \F, \mathbb{F}, \Prob)$, where $\mathbb{F}=\{\mathcal{H}_{t}\}_{t\geq0}$ is the right-continuous complete filtration generated by $(X,N^r)$. We assume that the processes $X$ and $N^r$ are independent.}
	
	We consider strategies $\pi := \left( L^{\pi}_{r}(t), R^{\pi}_{r}(t); t \geq 0 \right)\in\mathcal{A}$ where $L^{\pi}_{r}$ admits the form $\int_{[0,t]}v^{\pi}(s)\diff N^r(s)$, $t\geq0$, and $v^{\pi}$ is a c\`agl\`ad process adapted to the filtration $\mathbb{F}$. On the other hand, the process $R^{\pi}_{r}$ is nondecreasing, right-continuous and $\mathbb{F}$-adapted, with $R^{\pi}_{r}(0-)=0$ satisfying 
\begin{equation}\label{admissibility2}
\E_{x}\left[ \int_{[0,\infty)} e^{-qt} \diff R_r^{\pi}(t) \right] < \infty, \quad x \geq 0,
\end{equation}
and $U_r^{\pi}(t)\geq0$ a.s.. The  rate of discounting $q$ is a positive constant.

Let $\zeta$ be an exponential random variable with parameter $\lambda > 0$, independent of $X$, representing a random terminal time. We consider that a payoff is made upon termination, given by a function $w : [0,\infty) \rightarrow \R$. Then, assuming that $\beta > 1$ is the cost per unit of injected capital, the objective is to maximize the expected NPV
\begin{align} \label{v_pi}
		v_{\pi} (x) := \mathbb{E}_x \left[ \int_{[0, \zeta)} e^{-q t} \diff L_r^{\pi}(t) - \beta \int_{[0, \zeta)} e^{-q t} \diff R_r^{\pi}(t) + e^{-q \zeta} w(U_r^{\pi}(\zeta)) \right], \quad x\geq 0,
	\end{align}
	over  the set of all admissible strategies $\mathcal{A}$. Hence the problem is to compute the value function
	\begin{equation}\label{control:value}
		v(x):=\sup_{\pi \in\A}v_{\pi}(x), \quad x \geq 0,
	\end{equation}
	and obtain an optimal strategy $\pi^*$ such that $v_{\pi^{*}}=v$, if such a strategy exists.

	 {We make the following assumptions:
		\begin{assump}\label{assump_1_aux} We assume that $\E[X(1)] = \psi_X'(0+) > -\infty$.
		\end{assump}}
		\begin{assump}\label{assum_w}
		We assume that $w$ is {a concave function} with $$w'_{+}(0+) \leq \beta\quad \text{and}\quad w'_{+}(\infty):=\lim_{x\rightarrow\infty}w'_{+}(x)\ \in [0,1].$$ 
	\end{assump}
	\subsection{Spectrally negative processes with Parisian reflection above}\label{sec_snlp_par_ruin}
Let $\mathcal{T}_r = \lbrace T(i) : i \geq 1 \rbrace $ be the set of jump times of an independent Poisson process with rate $ r > 0$. We construct the \lev process with Parisian reflection above at  the level $b\geq 0$, denoted by ${U_r^b:=\{U_r^b(t):t\geq 0\}}$, as follows: the process is observed only at times  {belonging to the set} $\mathcal{T}_r$ and is  {pushed} down to the level $b$ if and only if it is  {observed} above $b$. Formally, we have:
\[
U_r^b(t) = X(t), \quad t \in [0, T_0^+(1)),
\]
where 
\begin{equation}\label{def_T_0_1}
T_b^+(1) := \inf \lbrace T(i) \in \mathcal{T}_r : X(T(i)) > b \rbrace.
\end{equation}
The process then jumps downward by $X(T_b^+(1))-b$ so that $U_r^{b}(T_b^+(1)) = b$. For $T_b^+(1) \leq t < T_b^+(2)  := \inf\{T(i) > T_b^+(1):\; U^{b}_r(T(i)-) > b\}$, we have $U_r^{b}(t) = X(t) - (X(T_b^+(1))-b)$.  The process $U_r^{b}$ can be constructed by repeating this procedure.

Suppose $L_r^b(t)$ is the cumulative amount of (Parisian) reflection until time $t \geq 0$. Then we have
\begin{align*}
	U_r^{b}(t) = X(t) - L_r^b(t), \quad t \geq 0,
\end{align*}
with
\begin{align}
	L_r^b(t) := \sum_{T_b^+(i) \leq t} \left(U_r^{b}(T_b^+(i)-)-b\right), \quad t \geq 0, \label{def_L_r}
\end{align}
where $(T_b^+(n); n \geq 1)$ can be constructed inductively by \eqref{def_T_0_1} and
\begin{eqnarray*}T_b^+(n+1) := \inf\{T(i) > T_b^+(n):\; U_r^{b}(T(i)-) >b\}, \quad n \geq 1.
\end{eqnarray*}


	\subsection{
		Periodic-classical barrier strategies}\label{PR}
	The objective of this section is to show the optimality of the periodic-classical barrier strategy 
	\[
	\bar{\pi}^{0,b} := \{(L_r^{0,b}(t), R_r^{0,b}(t)); t \geq 0 \}.
	\]
	The controlled process $U_r^{0,b}$ becomes the \textit{L\'evy process with Parisian reflection above and classical reflection below}, which can be constructed as follows.

	Let  $R^{{0,b}}_r (t):= (-\inf_{0 \leq s \leq t} X(s)) \vee 0$ for $t\geq 0$, then we have 
	\begin{align*}
		U_r^{0,b}(t) = X(t) + R^{{0,b}}_r(t), \quad 0 \leq t < \widehat{T}_b^{+} (1) 
	\end{align*}
	where $\widehat{T}_b^{+}(1) := \inf\{T(i):\; X(T(i)-)+R_r(T(i)-) > b\}$.
	The process then jumps down by $X(\widehat{T}_b^{+}(1))+R_r(\widehat{T}_b^{+}(1))-b$ 
	so that $U_r^{0,b}(\widehat{T}_b^{+}(1)) = b$. For $\widehat{T}_b^{+}(1) \leq t < \widehat{T}_b^{+}(2)  := \inf\{T(i) > \widehat{T}_b^{+}(1):\; U_r^{0,b}(T(i) -) > b\}$, $U_r^{0,b}(t)$ is the process reflected  at $0$ of  the process $( X(t) - X(\widehat{T}_b^{+}(1)) +b; t \geq \widehat{T}_b^+(1) )$. 
	The process $U_r^{0,b}$ can be constructed by repeating this procedure.
	It is clear that it admits a decomposition
	\begin{align*}
		U_r^{0,b}(t) = X(t) - L_r^{0,b}(t) + R_r^{0,b}(t), \quad t \geq 0,
	\end{align*}
	where $L_r^{0,b}(t)$ and $R_r^{0,b}(t)$ are, respectively, the cumulative amounts of Parisian and classical reflection until time $t$. 
	
	Notice that for $b \geq 0$, the strategy $\bar{\pi}^{0,b} := \{(L_r^{0,b}(t), R_r^{0,b}(t)); t \geq 0 \}$ is admissible for the  problem described at the beginning of this section, because \eqref{admissibility2} holds by  Proposition \ref{NPV_aux} and Assumption \ref{assump_1_aux}. Its expected NPV of dividends minus the costs of capital injection and payoff at an exponential time is denoted by 
	\begin{align} \label{v_pi2}
		v_b(x) :&= \mathbb{E}_x \left[ \int_{[0, \zeta)} e^{-q t} \diff L_r^{0,b}(t) - \beta \int_{[0, \zeta)} e^{-q t} \diff R^{0,b}_r(t) + e^{-q \zeta} w(U_r^{0,b}(\zeta)) \right] \\
		& = \mathbb{E}_x \left[ \int_{[0, \infty)} e^{-\theta t} \diff L_r^{0,b}(t) - \beta \int_{[0, \infty)} e^{-\theta t} \diff R^{0,b}_r(t) + \lambda \int_0^{\infty} e^{- \theta t} w(U_r^{0,b}(t)) \diff t \right], \quad x\geq 0, \notag
	\end{align}
	where $\theta := q + \lambda$. 
	
	The main result for this section confirms the optimality of the periodic-classical barrier strategy for the auxiliary control problem.
	\begin{theorem}
Under Assumptions {\ref{assump_1_aux}} and \ref{assum_w}, there exists a constant barrier $0 \leq b^* < \infty$ such that the periodic-classical reflection strategy at the threshold $b^*$ is optimal, i.e., $\overline{\pi}^{0,b^*}$ is an optimal strategy for the problem \eqref{control:value} and the value function is given by 
\[
v(x) = v_{\overline{\pi}^{0,b^*}}(x) = v_{b^*}(x), \quad \text{ for } x \geq 0.
\]
	\end{theorem}

\section{Expression of $v_b$ using the scale function.}\label{sec_NPV_Scale}
In this section we will write an expression for the expected NPV of total costs $v_b$ as in \eqref{v_pi2}. For convenience, let us denote
\[
v_b^{LR}(x) = \mathbb{E}_x \left[ \int_{[0, \infty)} e^{-\theta t} \diff L_r^{0,b}(t) -\beta \int_{[0, \infty)} e^{-\theta t} \diff R_r^{0,b}(t) \right], \, v_b^w(x) = \mathbb{E}_x \left[ \int_{[0, \infty)} e^{-\theta t} w(U_r^{0,b}(t))  \diff t \right].
\]
 {It is clear} that $v_b(x) = v_b^{LR}(x) + \lambda v_b^w(x)$  {for $x\geq0$}. We also have that the expected NPV of dividend payments and capital injection $v_b^{LR}$ has already been computed in Lemma 3.1 in \cite{NobPerYamYan} 
, which is  given by
 \begin{align}\label{v1}
 v^{LR}_{b}(x)&=- C^1_b\left( Z_b^{(\theta,r)}(x) - r Z^{(\theta)}(b) \overline{W}^{(r+\theta)}(x-b) \right) - r \overline{\overline{W}}^{(r+\theta)}(x-b)\\
 &\quad + \beta \left( \overline{Z}_b^{( \theta ,r)}(x) + \dfrac{\psi_X'(0+)}{\theta} - r \overline{Z}^{(\theta)}(b) \overline{W}^{(\theta+r )}(x-b) \right),\notag
 \end{align}
 with
\begin{equation}\label{v2}
C^1_b = \dfrac{r ( \beta Z^{(\theta)}(b) -1)}{\theta \Phi(r+\theta) Z^{(\theta)}(b,\Phi(r+\theta))} + \dfrac{\beta}{\Phi(r+\theta)}.
\end{equation}
Therefore, it only remains to compute the expected NPV of running costs $v_b^w$ . To this end, we provide the following result and the proof is deferred to Appendix \ref{resol_par_app} (see Section \ref{App_Proof_1}).
 \begin{proposition}\label{resol_par}
 	For $x,b\in\R$, $q > 0$, and a positive measurable function $h$ on $\R$ with compact support,
 	\begin{align}\label{r12}
 		g^{(q)}(x;h)&:=\E_x\left[\int_0^{\infty}e^{-qt}h(U_r^{0,b}(t))dt\right]\\
 		&=\dfrac{(C^{(q,r)}(b;h)+ {\rho^{(q)}_{b}(b;h)})}{Z^{(q)}(b)}Z^{(q,r)}_{b}(x)- {\rho^{(q,r)}_{b}(x;h)}\notag\\
 		& \quad -rC^{(q,r)}(b;h)\overline{W}^{(q+r)}(x-b)-\int_0^{x-b}h(z+b)W^{(q+r)}(x-b-z)\diff z, \notag
 	\end{align}
 	where
 	\begin{align}\label{e1}
 		\begin{split}
 			C^{(q,r)}(b;h)&=\dfrac{\Xi^{(q,r)}(b;h)-r\dfrac{\rho^{(q)}_{b}(b;h)}{Z^{(q)}(b)}\displaystyle\int_b^{\infty}e^{-\Phi(q+r) (y-b) }Z^{(q)}(y)\diff y}{\displaystyle\dfrac{r}{Z^{(q)}(b)}\int_b^{\infty}e^{-\Phi(q+r) (y-b) }Z^{(q)}(y)\diff y-\dfrac{r}{\Phi(q+r)}}.
 		\end{split}
 	\end{align}
 \end{proposition}
Now we provide an expression of the expected NPV of the periodic-classical barrier strategy with additional running costs $v_b$, given in \eqref{v_pi2}, in terms of scale functions. We omit the proof as it is a direct consequence of \eqref{v1} and  Proposition \ref{resol_par}, due to the fact that
\begin{align}\label{v_w} 
v_b^w(x) = g^{(\theta)}(x;w)
&=\dfrac{(C^{(\theta,r)}(b;w)+ {\rho^{(\theta)}_{b}(b;w)})}{Z^{(\theta)}(b)}Z^{(\theta,r)}_{b}(x)- {\rho^{(\theta,r)}_{b}(x;w)}\\
& \quad -rC^{(\theta,r)}(b;w)\overline{W}^{(\theta+r)}(x-b)-\int_0^{x-b}w(z+b)W^{(\theta+r)}(x-b-z)\diff z.\notag
\end{align}
\begin{proposition}\label{NPV_aux}
For $b\geq0$ and $x\in\R$,
\begin{align}\label{Aux_NPV}
	v_b(x) &= - C^1_b\left( Z_b^{(\theta,r)}(x) - r Z^{(\theta)}(b) \overline{W}^{(r+\theta)}(x-b) \right) - r \overline{\overline{W}}^{(\theta + r)}(x-b)\\
		&\quad + \beta \left( \overline{Z}_b^{( \theta ,r)}(x) + \dfrac{\psi'(0+)}{\theta} - r \overline{Z}^{(\theta)}(b) \overline{W}^{(r+ \theta )}(x-b) \right)\notag \\
		&\quad + \lambda \bigg[ \dfrac{( {C^{(\theta,r)}(b;w)} + \rho_b^{(\theta)}(b;w))}{Z^{(\theta)}(b)} Z_b^{(\theta,r)}(x)- \rho_b^{(\theta,r)}(x;w)   \notag  \\
		& \qquad   - r  {C^{(\theta,r)}(b;w)} \overline{W}^{(r + \theta)}(x-b)   -\int_0^{x-b} w(z+b) W^{(r+\theta)}(x-b-z) \diff z \bigg],\notag
\end{align} 
where  {$C^{1}_{b}$ and $C^{(\theta,r)}(b;w)$ are given in  \eqref{v2}  and  \eqref{e1} respectively}.

\end{proposition}

Additionally, in the next result we provide an expression for the resolvent of L\'evy process with Parisian reflection above at the threshold $b$, $U_r^{b}$. The proof is deferred to the Appendix (see Section \ref{App_Proof_2}).

\begin{proposition}\label{resol_par_ruin}
	For $x,b\in\R$, $q>0$, and a positive measurable function $h$ on $\R$ with compact support, 
	\begin{align}\label{r12aa}
	\tilde{g}^{(q)}(x;h)  &:=\E_x\left[\int_0^{\tau_0^-(r)}e^{-qt}h(U_r^{b}(t))dt\right]\\
		&= \dfrac{(\tilde{C}^{(q,r)}(b;h)+\rho^{(q)}_{b}(b;h))}{W^{(q)}(b)} W^{(q,r)}_b(x) - \rho^{(q,r)}_b(x;h)\notag\\
		& \quad - r \tilde{C}^{(q,r)}(b;h)\overline{W}^{(q+r)}(x-b)  - \int_0^{x-b} h(y+b) W^{(q+r)}(x-b-y) \diff y,\notag
	\end{align}
	where
	\begin{equation}\label{r11a1}
		\tilde{C}^{(q,r)}(b;h)= \dfrac{\Xi^{(q,r)}(b;h)-\dfrac{\rho^{(q)}_{b}(b;h)}{W^{(q)}(b)} Z^{(q)}(b;\Phi(q+r)) }{\dfrac{ Z^{(q)}(b;\Phi(q+r)) }{W^{(q)}(b)}-\dfrac{r}{\Phi(q+r)}}.
	\end{equation}
\end{proposition}

\section{Selection of a candidate optimal barrier $b^*$}\label{sec_candidate_optimal}
We focus on the periodic barrier strategy defined at the beginning of Section \ref{PR} and choose the candidate optimal barrier $b^*$, which satisfies  that  $v_{b^*}'(b^*) = 1$ if such $b^* > 0$ exists, and set it to be 0 otherwise.

Recall that the expected NPV of the periodic-classical barrier strategy is given by expression \eqref{Aux_NPV}. {We first} analyse the smoothness of the function $v_b$. 
\begin{lemma}\label{lem_NPV_diff} For all $b\geq 0$, and {$x \in \R \setminus \{0,b\}$},
\begin{align}\label{NPV_diff}
	v'_b(x) &= -\theta C^1_b W_b^{(\theta,r)}(x) - r \overline{W}^{(\theta + r)}(x-b) + \beta Z_b^{(\theta,r)}(x) \\
	& \quad + \lambda \left[  \left( \theta\dfrac{{C^{(\theta,r)}(b;w)} + \rho_b^{(\theta)}(b;w)}{Z^{(\theta)}(b)} - w(0) \right) W_b^{(\theta,r)}(x) \right. \notag\\
	&\qquad \left.  - \left( \int_0^{x-b} w'_{+}(z+b) W^{(r+\theta)}(x-b-z) \diff z + \rho_b^{(\theta,r)}(x;w'_{+}) \right) \right],\notag
\end{align}
and, if $X$ has paths of unbounded variation, 
\begin{align}\label{NPV_ddiff}
	v''_b(x) &= - \theta C^1_b \left( W^{(\theta)\prime}(x) + r W^{(\theta +r)}(x-b) W^{(\theta)}(b) + r\int_b^x W^{(\theta + r)}(x-y) W^{(\theta)\prime}(y) \diff y \right)\\
	&  - r W^{(\theta+r)}(x-b) + \beta \left( \theta W_b^{(\theta,r)}(x) + r W^{(\theta+r)}(x-b) Z^{(\theta)}(b)  \right) \notag\\
	&  + \lambda \left[ \left(\theta  \dfrac{{C^{(\theta,r)}(b;w)} + \rho_b^{(\theta)}(b;w)}{Z^{(\theta)}(b)}-w(0) \right) \right.\notag\\
	&\times\left( W^{(\theta)\prime}(x) + r W^{(\theta +r)}(x-b) W^{(\theta)}(b) + \int_b^x W^{(\theta + r)}(x-y) W^{(\theta)\prime}(y) \diff y   \right)  \notag\\
	&  -  \left( \int_0^b W^{(\theta)\prime}(x-y) w'_{+}(y) \diff y + r  \int_b^x W^{(\theta+r)\prime}(x-y)  \rho_b^{(\theta)}(y;w'_{+}) \diff y \right) \notag\\
	&\left. - \int_0^{x-b} w'_{+}(z+b) W^{(\theta+r)\prime}(x-b-z) \diff z  \right].\notag
\end{align}
\end{lemma}
\begin{proof}
{The first and second (if $X$ has paths of unbounded variation) derivatives of  $v_b^{LR\,\prime}$ are computed in Lemma 3.2 in \cite{NobPerYamYan}. Hence, it remains} to compute the first and second derivatives of   $v_b^w$.

{Differentiating  $Z^{(\theta,r)}_{b}(x)$ and $\rho^{(\theta,r)}_{b}(x,w)$ (given in \eqref{identities} and \eqref{fun_rho_0} respectively), using \eqref{eq1} and integration by parts}, it can be checked that 
	for $x >0$ we have
\begin{align*}
	Z_b^{(\theta,r)\prime}(x) &= \theta W_b^{(\theta,r)}(x) + r Z^{(\theta)}(b) W^{(\theta+r)}(x-b),\\
	\rho_{b}^{(\theta, r)\prime} (x ;w) &= {W}^{(\theta)}(x)w(0) - {W}^{(\theta)}(x-b) w(b) + \rho^{(\theta)}_{b}(x;w'_{+})\\
	&\quad+r\bigg(W^{(\theta+r)}(0)\rho^{(\theta)}_{b}(x;w)+ \int_b^{x } W^{(\theta + r)\prime} (x-y) \rho^{(\theta)}_{b} (y;w) \diff y\bigg) \notag\\
	&= \rho_b^{(\theta,r)}(x,w'_{+}) + r W^{(\theta+r)}(x-b) \rho_{b}^{(\theta)}(b;w) + w(0) W_b^{(\theta,r)}(x) - W^{(\theta+r)}(x-b)w(b). \notag
\end{align*}
From here, we obtain by {differentiating \eqref{v_w}}
\begin{align}\label{resolvent_dif}
v_b^{w \, \prime}(x) &= \dfrac{{C^{(\theta,r)}(b;w)} + \rho_b^{(\theta)}(b;w)}{Z^{(\theta)}(b)}\left( \theta W_b^{(\theta,r)}(x) + r Z^{(\theta)}(b) W^{(\theta+r)}(x-b) \right) - r {C^{(\theta,r)}(b;w)} W^{(\theta+r)}(x-b) \\
& \quad - \left( \rho_b^{(\theta,r)}(x,w'_{+}) + r W^{(\theta+r)}(x-b) \rho_{b}^{(\theta)}(b;w) + w(0) W_b^{(\theta,r)}(x) - W^{(\theta+r)}(x-b)w(b) \right) \notag\\
&\quad - \left( w(b)W^{(r+\theta)}(x-b) + \int_0^{x-b} w'_+(z+b) W^{(\theta+r)}(x-b-z) \diff z \right) \notag\\
&=\left( \theta \dfrac{{C^{(\theta,r)}(b;w)} + \rho_b^{(\theta)}(b;w)}{Z^{(\theta)}(b)}  - w(0) \right) W_b^{(\theta,r)}(x)  - \rho_b^{(\theta,r)}(x;w'_{+}) \notag\\
&\quad-  \int_0^{x-b} w'_{+}(z+b) W^{(\theta+r)}(x-b-z) \diff z , \notag
\end{align}
{where we used a change of variable to compute the derivative of  $\int_0^{x-b} w(z+b) W^{(\theta+r)}(x-b-z) \diff z$ given in the last term of the first equality}.

In addition, when $X$ has paths of unbounded variation, {we have using \eqref{identities},  \eqref{fun_rho_0} and Remark \ref{remark_scale_function_properties}(2)}
\begin{align}\label{resolvent_dif_2}
W_b^{(\theta,r)\prime}(x)&=W^{(\theta)\prime}(x) + r W^{(\theta +r)}(x-b) W^{(\theta)}(b) + \int_b^x W^{(\theta + r)}(x-y) W^{(\theta)\prime}(y) \diff y,\notag\\  
\rho^{(\theta,r)\prime}_{b}(x;w'_{+})&=\rho^{(\theta)\prime}_{b}(x;w'_{+})+r\int_{b}^{x}W^{(\theta)\prime}(x-y)\rho^{(\theta)}_{b}(y;w'_{+})\diff y.
\end{align}
Then, {by differentiating \eqref{resolvent_dif} and using \eqref{resolvent_dif_2} together with Remark \ref{remark_scale_function_properties}(2)}
\begin{align}
\label{a001}
v_b^{w \, \prime \prime}(x)
&= \left(\theta \dfrac{{C^{(\theta,r)}(b;w)} + \rho_b^{(\theta)}(b;w)}{Z^{(\theta)}(b)}-w(0)\right)\\
&\quad\times \left( W^{(\theta)\prime}(x) + r W^{(\theta +r)}(x-b) W^{(\theta)}(b) + \int_b^x W^{(\theta + r)}(x-y) W^{(\theta)\prime}(y) \diff y   \right) \notag\\
&\quad - \left( \int_0^b W^{(\theta)\prime}(x-y) w'_{+}(y) \diff y + r  \int_b^x W^{(\theta+r)\prime}(x-y)  \rho_b^{(\theta)}(y;w'_{+}) \diff y \right) \notag\\
&\quad  - \int_0^{x-b} w'_{+}(z+b) W^{(\theta+r)\prime}(x-b-z) \diff z . \notag
\end{align}
The proof is completed.
\end{proof}
\begin{remark}\label{Remark_C1_Resolvent}
{Using \eqref{lem_NPV_diff} and \eqref{NPV_ddiff} together with Remark \ref{remark_scale_function_properties}(1) we have that the mapping $x\mapsto v_b^w(x)$ is continuously differentiable (resp. twice continuously differentiable) on $\mathbb{R}\backslash\{0,b\}$ when $X$ is of bounded variation (resp. unbounded variation).	
In addition, we have by \eqref{resolvent_dif} that
\[
v_b^{w \, \prime}(b+) = \left( \theta \dfrac{ {C^{(\theta,r)}(b;w)} + \rho_b^{(\theta)}(b;w)}{Z^{(\theta)}(b)}  - w(0) \right) W^{(\theta)}(b)  - \rho_b^{(\theta)}(b;w') = v_b^{w \, \prime}(b-).
\]
On the other hand, if $X$ is of unbounded variation, by \eqref{a001} and Remark \ref{remark_scale_function_properties}(2)
\[
v_b^{w \, \prime\prime}(b+)=\left(\theta \dfrac{{C^{(\theta,r)}(b;w)} + \rho_b^{(\theta)}(b;w)}{Z^{(\theta)}(b)}-w(0)\right)W^{(\theta)\prime}(b) - \int_0^b W^{(\theta)\prime}(b-y) w'_{+}(y) \diff y=v_b^{w \, \prime\prime}(b-).
\]}
\end{remark}
By the smoothness of the scale function, together with Lemma 3.3 from \cite{NobPerYamYan} and Remark \ref{Remark_C1_Resolvent} we obtain the following result.
\begin{lemma}[Smoothness of $v_b$]\label{lem_NPV_smooth}
	For all $b \geq 0$ we have:
	\begin{enumerate}
		\item[(i)] When $X$ has paths of bounded variation, $v_b$ is continuously differentiable on $\R \setminus \{0\}$;
		\item[(ii)] When $X$ has paths of unbounded variation, $v_b$ is twice continuously differentiable on $\R \setminus \{0\}$.
	\end{enumerate}
\end{lemma}
\subsection{Selection and existence of the optimal barrier}

In this section, we will define and {prove} the existence of {the threshold $b^{*}\geq 0$  under which the strategy} ${\bar{\pi}}^{b^{*}}=(L^{0,b^{*}}_{r}(t),R^{0,b^{*}}_{r}(t))$ is optimal. For that {purpose, we provide} some preliminary results.

{\begin{remark}\label{par_ruin}
\begin{enumerate}
	\item Fix $b \geq 0$. Let
	$U^{b}_{r}$ be the Parisian reflected process of $X$ from above at $b$ (without classical reflection) as given in  Section \ref{sec_snlp_par_ruin}, and 
	\[
	\tau_{0}^-(r):=\inf\{t>0: U^{b}_{r}(t)<0 \}.
	\]
	Then, by Corollary 3 in \cite{PerYam2}, for any $x\in\mathbb{R}$ and $q>0$, 
	\begin{align}\label{403}
	\mathbb{E}_{x}\Big[e^{-q\tau_{0}^-(r)}\Big]
	&=Z_{b}^{(q,r)}(x)-rZ^{(q)}(b)\overline{W}^{(q+r)}(x-b)\\
	&\quad-q\dfrac{Z^{(q)}(b,\Phi(q+r))}{Z^{(q)\prime}(b,\Phi(q+r))}\left( W_{b}^{(q,r)}(x)-rW^{(q)}(b) \overline{W}^{(q+r)}(x-b)\right),\notag
	\end{align}
	where in particular
	\begin{align}\label{fpt_pr}
	\mathbb{E}_b\left[e^{-q\tau_{0}^-(r)}\right]
	&=Z^{(q)}(b) -q\dfrac{Z^{(q)}(b,\Phi(q+r))}{Z^{(q)\prime}(b,\Phi(q+r))}W^{(q)}(b).
	\end{align}
	\item 	{By} (4.1) and Remark 4.1 from \cite{NobPerYamYan}, 
	\begin{equation}\label{eq3}
	v_b^{LR \, \prime}(b) = \dfrac{\theta \left( \beta \E_b \left[ e^{-\theta \tau_{0}^-(r)} \right] -1 \right)}{\Phi(\theta+r) \left( Z^{(\theta)}(b) - \E_b[ e^{ -\theta \tau_{0}^-(r) } ] \right)} W^{(\theta)}(b) + 1.
	\end{equation}
	\end{enumerate}
\end{remark}}

We now provide the following auxiliary result and we defer the proof to Appendix \ref{useful_iden}.
\begin{lemma}\label{useful_iden_lemma}
For $b\geq0$ we have
	\begin{align}\label{r13}
		v_b^{w\,\prime}(b)&=\theta W^{(\theta)}(b) \left( \dfrac{{C^{(\theta,r)}(b;w)}+ \rho_b^{(\theta)}(b;w)}{Z^{(\theta)}(b)} \right) -w(0)W^{(\theta)}(b)-\rho^{(\theta)}_{b}(b;w'_{+}) \\ 
		&=\dfrac{ \theta \E_b \left[\displaystyle \int_0^{\tau_0^-(r)} e^{-\theta t}w'_{+}(U_r^{b}(t)) \diff t \right]}{ \Phi(\theta+r) \left( Z^{(\theta)}(b) - \E_b[ e^{ -\theta \tau_{0}^-(r) } ] \right) } W^{(\theta)}(b).\notag
	\end{align}
\end{lemma}
{Using that $v_{b}'(b)=v_b^{LR \, \prime}(b)+\lambda v_b^{w\,\prime}(b)$ together with} \eqref{eq3} and \eqref{r13}, gives 
\begin{equation}\label{NPV_diff_b}
	v_b'(b) = \dfrac{\theta}{\Phi(\theta+r)} \dfrac{G(b)}{Z^{(\theta)}(b) - \E_b\left[ e^{-\theta \tau_{0}^{-}(r)} \right]} W^{(\theta)}(b)  + 1,
\end{equation}
where
\begin{align}\label{root_fun}
	G(b) &:= \beta - 1 - \E_b \left[ \int_0^{\tau_0^-(r)} e^{-\theta t} \left( \beta \theta - \lambda w'_{+}(U_r^{b}(t)) \right) \diff t \right]\\
	&=\left( \beta  Z^{(\theta)}(b) - 1 + \lambda \tilde{C}^{(\theta,r)}(b;w'_{+}) \right) - \beta \theta W^{(\theta)}(b)\dfrac{Z^{(\theta)}(b,\Phi(\theta+r))}{Z^{(\theta)\prime}(b,\Phi(\theta+r))}\quad b \geq 0,\notag
\end{align}
where the last equality is true because of \eqref{r12aa} {(taking $w'_+$ instead of $h$) and \eqref{fpt_pr}}. We propose as candidate for the optimal barrier 
\begin{align}\label{opt_barr}
	b^*:=\inf\{b\geq0: G(b)\leq 0\}.
\end{align}
{In the next result we provide a necessary and sufficient condition for the optimal barrier $b^*$ to be $0$.}
\begin{proposition}\label{prop_existence} We have that $0 \leq b^* < \infty$. Moreover, we have that $b^* = 0$ if and only if $X$ has paths of bounded variation and
	\begin{equation}\label{cond_exist_opt}
	\beta - 1 \leq \dfrac{1}{c}\left( \dfrac{1}{\Phi(\theta + r) - \dfrac{r}{c}} \right) \left( \theta\beta - \lambda \Phi(\theta + r) \int_0^{\infty} e^{-\Phi(\theta+r)z} w'_{+}(z) \diff z \right).
	\end{equation}
\end{proposition}
\begin{proof}
	Due to Assumption \ref{assum_w} we have that $b \mapsto \beta \theta - \lambda w'_{+}(b)$ is non-decreasing, hence 
	\[
	b \mapsto \E_b \left[ \int_0^{\tau_0^-(r)} e^{-\theta t} \left( \beta \theta - \lambda w'_{+}(U_r^{b}(t)) \right) \diff t \right] 
	\]
	 is non-decreasing as well. It follows that {the mapping $b\mapsto G (b)$} is non-increasing.
	
	
	On the other hand, due to spatial homogeneity of L\'evy processes, we have that the $\lbrace U_r^b (t) ; t \leq \tau_0^-(r) \rbrace$ started at $U_r^b(0) = b$ is equal in law to $\lbrace b + U_r^0 ; t \leq \tilde{\tau}_{-b}^-(r) \rbrace$ started at $U_r^0(0) = 0$, where $\tilde{\tau}_{-b}^-(r) := \inf \lbrace t \geq 0: U^{0}_r(t) < -b \rbrace$. Then, by {dominated convergence} we have
	\begin{align}
		\lim_{b \uparrow \infty} G(b) &= \beta-1- \lim_{b \uparrow \infty} \E_0 \left[ \int_0^{\infty} e^{-\theta t} 1_{\lbrace t \leq \tilde{\tau}_{-b}^-(r) \rbrace} \left[\beta\theta-\lambda w'_{+}(b + U^{0}_r(t))\right] \diff t \right] = \dfrac{\lambda}{\theta} w'_{+}(+\infty) - 1 < 0. \label{opt_lim1}
	\end{align}
Now,  using \eqref{root_fun}, we obtain
\begin{align*}
	\lim_{b \downarrow 0} G(b) &= \beta \left( 1 - \theta W^{(\theta)}(0+) \dfrac{1}{\Phi(\theta+r) - r W^{(\theta)}(0+)} \right) - 1 \\
	& \quad +  \dfrac{ \lambda\Phi(\theta+r) W^{(\theta)}(0+)}{\Phi(\theta+r) - r W^{(\theta)}(0+)} \int_0^{\infty} e^{-\Phi(\theta+r)z} w'_{+}(z) \diff z ,
\end{align*}
where we have used {that} $\lim_{b \downarrow 0} \rho_b^{(\theta)}(b;w'_{+}) = 0$  and $\lim_{b \downarrow 0} \Xi^{(\theta,r)}(b;w'_{+}) = \int_{0}^{\infty}e^{-\Phi(q+r)y}w'_{+}(z)\diff z$. Hence, we get the following cases:
\begin{enumerate}
	\item If $X$ has paths of unbounded variation, then $\lim_{b \downarrow 0} G(b) = \beta - 1 > 0$. Thus, there exists a unique $b^* > 0$ such that $G(b^*) = 0$.
	\item If $X$ has paths of bounded variation we have
	\[
	\lim_{b \downarrow 0} G(b) = \beta - 1 - \dfrac{1}{c}\left( \dfrac{1}{\Phi(\theta + r) - \dfrac{r}{c}} \right) \left( \theta \beta - \lambda \Phi(\theta + r) \int_0^{\infty} e^{-\Phi(\theta+r)z} w'_{+}(z) \diff z \right).
	\]
	Thus, if \eqref{cond_exist_opt} does not hold, then there  exists a unique $b^* > 0$ such that $G(b^*) = 0$; otherwise, if \eqref{cond_exist_opt} holds, we set 
	$b^* = 0$.
\end{enumerate}
\end{proof}
\section{Verification of Optimality}\label{verification_aux}
We shall show the optimality of the periodic-classical barrier strategy $\bar{\pi}^{0,b^*}$, where the barrier $b^*$ is defined by \eqref{opt_barr}.
\begin{theorem}\label{thm_optim_aux}
	The strategy $\bar{\pi}^{0,b^*}$ is optimal and the value function of the problem \eqref{control:value} is given by $v = v_{b^*}$.
\end{theorem}
Let $\gen$ be the infinitesimal generator associated with the process $X$ applied to a
{$\mathcal{C}^1(0,\infty)$ (resp., $\mathcal{C}^2(0,\infty)$)} function $f$ for the case $X$ is of bounded (resp., unbounded) variation:
\begin{align}\label{diff1}
	\gen f(x)\eqdef \gamma f'(x)+\dfrac{1}{2}\eta^{2}f''(x)+\int_{(-\infty,0)}[f(x+z)-f(x)-f'(x)z1_{\{-1<z<0\}}]\Pi(\diff z)\ \quad{for}\ x>0.
\end{align}

{In the next result we provide a verification lemma. The proof is essentially the same as Proposition 5.1 in \cite{NobPerYamYan} (which deals with the case in which the payoff function $w$ is equal to zero), and thus we omit it.} 

Throughout the rest of this section we extend the domain of $v_{\pi}$ to $\R$ by setting $v_{\pi}(x) = \beta x + v_{\pi}(0+)$ for $x < 0$.
\begin{lemma}[Verification lemma]\label{prop_HJB_aux}
	Suppose that $\hat{\pi} \in \A$ is such that $v_{\hat{\pi}}\in\mathcal{C}(\R)\cap\mathcal{C}^1((0,\infty))$ (respectively, $\mathcal{C}^{1}(\R)\cap\mathcal{C}^2((0,\infty))$) for the case that $X$ has paths of bounded (respectively, unbounded) variation. In addition, suppose that
	\begin{equation}\label{HJB_aux}
		\begin{split}
		(\gen - \theta) v_{\hat{\pi}}(x) + r \max_{0 \leq l \leq x} \lbrace l + v_{\hat{\pi}}(x-l) - v_{\hat{\pi}}(x) \rbrace + \lambda w(x) &\leq 0, \qquad x \geq 0,\\ v'_{\hat{\pi}}(x) &\leq \beta, \qquad x \geq 0, \\
		\inf_{x \geq 0} v_{\hat{\pi}}(x) &> -m \text{ for some } m > 0.
		\end{split}
	\end{equation}
	Then $\hat{\pi}$ is an optimal strategy and $v_{\hat{\pi}}(x) = v(x)$ for all $x \geq 0$. 
\end{lemma}

{Notice that if $v_{b^{*}}$ satisfies the variational inequalities \eqref{HJB_aux}, then the strategy $\bar{\pi}^{0,b^{*}}$ is optimal, due to the previous lemma. To show this, we shall provide some preliminary results.}
\begin{lemma}\label{lem_HJB_1}
For $b \geq 0$, we have
\begin{equation}\label{HJB1}
(\gen - \theta) v_b(x) =
\begin{cases}
	-\lambda w(x), & \text{if } x \in (0,b), \\
	-r\lbrace (x-b) + v_b(b) - v_b(x) \rbrace - \lambda w(x), & \text{if } x \in [b,\infty).
\end{cases}
\end{equation}
\end{lemma}
\begin{proof}
	First, from Lemma 5.1 in \cite{NobPerYamYan} we have
	\begin{equation}\label{HJB_LR}
		(\gen - \theta) v^{LR}_b(x) =
		\begin{cases}
			0, & \text{if } x \in (0,b), \\
			-r\lbrace (x-b) + v^{LR}_b(b) - v^{LR}_b(x) \rbrace , & \text{if } x \in [b,\infty).
		\end{cases}
	\end{equation}
It remains to analyse the term $(\gen - \theta) v^{w}_b(x)$ {for $x\in(0,\infty)$}. \\
(i) Suppose $0 < x < b$. {By the proof of Theorem 2.1 in \cite{BKY}} it follows that
\begin{equation}\label{HJB2}
(\gen - \theta) Z^{(\theta)}({x}) = 0, \quad {x} > 0.
\end{equation}
In addition, {the proof of Lemma 4.5 of \cite{EgaYam} implies}
\begin{equation}\label{HJB3}
	(\gen - \theta) \rho^{(q)}_{b}(x;w) = w(x), \quad 0 < x < b. 
\end{equation}
By combining \mazenta{\eqref{v_w},} \eqref{HJB2} and \eqref{HJB3}, we obtain
\begin{equation}\label{HJB4}
	(\gen - \theta) v^{w}_b(x) = -w(x), \quad 0 < x < b.
\end{equation}
(ii) {Now, assume that} $x > b$. From (5.4) and (5.6) in \cite{NobPerYamYan} we have
\begin{align}
\begin{split}
	(\gen - \theta)& \overline{W}^{(\theta+r)}(x-b) = 1 + r\overline{W}^{(\theta+r)}(x-b), \\
	(\gen - \theta)& Z_b^{(\theta,r)}(x) = rZ_b^{(\theta,r)}(x). \label{HJB5}
	\end{split}
\end{align}
On the other hand, from the proof of Lemma 4.5 of \cite{EgaYam} we have
\begin{align*}
(\gen - (\theta + r))\int_0^{x-b} W^{(\theta + r)}(x-b-y) w(y+b) \diff y = w(x).
\end{align*}
Hence{
\begin{equation}\label{HJB7}
	(\gen - \theta) \left(\int_0^{x-b} W^{(\theta + r)}(x-b-y) w(y+b) \diff y \right) = w(x) + r \int_0^{x-b} W^{(\theta + r)}(x-b-y) w(y+b) \diff y.
\end{equation}}
In a similar way, we obtain {by \eqref{fun_rho_0}
\begin{align*}
(\gen - \theta) \left(\int_b^x W^{(\theta + r)}(x-y) \rho^{(\theta)}_b(y;w) \diff y \right) &= \rho^{(\theta)}_b(x;w)+r\int_b^x W^{(\theta + r)}(x-y) \rho^{(\theta)}_b(y;w) \diff y\notag\\&=\rho_b^{(\theta,r)}(x;w).
\end{align*}}
{By the proof of Lemma 4 in \cite{AvrPalPis} we have that  $(\gen - \theta) W^{(\theta)}(x-y)=0$ for $x>y>0$, then by dominated convergence we get
\[
(\gen - \theta) \rho^{(\theta)}_b(x;w) = (\gen - \theta)\int_0^bW^{(\theta)}(x-y)w(y) \diff y=0, \quad x > b.
\]
Therefore, using \eqref{fun_rho_0}}
\begin{equation}\label{HJB8}
	(\gen - \theta) \rho_b^{(\theta,r)}(x;w) = r \rho_b^{(\theta,r)}(x;w).
\end{equation}
Finally, by {\eqref{v_w} together with} \eqref{HJB5}--\eqref{HJB8} we obtain
\begin{align}\label{HJB_w}
	(\gen - \theta) v^{w}_b(x) &= r \left( \dfrac{{C^{(\theta,r)}(b;w)} + \rho_b^{(\theta)}(b;w)}{Z^{(\theta)}(b)} \right) Z_b^{(\theta,r)}(x)  - r {C^{(\theta,r)}(b;w)}(1 + r\overline{W}^{(\theta+r)}(x-b))\\
	& \quad - r \rho_b^{(\theta,r)}(x;w)- \left(  w(x) + r \int_0^{x-b} W^{(\theta + r)}(x-b-y) w(y+b) \diff y \right) \notag \\
	& = -r \left( v^{w}_b(b) - v^{w}_b(x) \right) - w(x).  \notag 
\end{align}
The result follows by combining \eqref{HJB_LR}, \eqref{HJB4} and \eqref{HJB_w}.
\end{proof}

\begin{lemma}\label{lem_HJB_2}
We have $1 \leq v'_{b^*}(x) \leq \beta$ for $x \in (0, b^*)$ and $0 \leq v'_{b^*}(x) \leq 1$ for $x \in (b^*, \infty)$.
\end{lemma}
\begin{proof}
(i) Suppose that $b^* > 0$. {First we note that using \eqref{NPV_diff} and the choice of the optimal threshold $b^*$ (such that $v_{b^*}(b^*)=1$; see \eqref{NPV_diff_b}) we obtain
\begin{align}\label{der_AUX_NPV}
	v_{b^*}'(b^*) &= - \theta C^1_{b^{*}} W^{(\theta)}(b^*) + \beta Z^{(\theta)}(b^*) \\
	&\quad+ \lambda \left[\left( \theta \dfrac{{C^{(\theta,r)}(b^*;w)} + \rho_{b^{*}}^{(\theta)}(b^*;w)}{Z^{(\theta)}(b^*)} -w(0)\right)W^{(\theta)}(b^*) -   \rho^{(\theta)}_{b^*}(b^*;w'_{+}) \right]=1 .\notag
\end{align}
Then, using \eqref{der_AUX_NPV} in \eqref{NPV_diff} implies}
\[
\begin{split}
	v'_{b^*}(x) &=
	 W_{b^*}^{(\theta,r)}(x) \left( -\theta C^1_{b^*} + \lambda \left(\theta \dfrac{{C^{(\theta,r)}(b^*;w)} + \rho_{b^*}^{(\theta)}(b^*;w)}{Z^{(\theta)}(b^*)} -w(0) \right) \right) \\ 
	 &\quad- r \overline{W}^{(\theta + r)}(x-{b^*}) + \beta Z_{b^*}^{(\theta,r)}(x)  \\
	&\quad+ \lambda \left[ -  \int_0^{x-{b^*}} w'_{+}(z+{b^*}) W^{(r+\theta)}(x-{b^*}-z) \diff z - \rho_{b^*}^{(\theta,r)}(x;w'_{+}) \right]  \\
	& = W_{b^*}^{(\theta,r)}(x) \left( \dfrac{1 - \beta Z^{(\theta)}(b^*)}{W^{(\theta)}(b^*)} \right)  - r \overline{W}^{(\theta + r)}(x-{b^*}) + \beta Z_{b^*}^{(\theta,r)}(x) \\
	&\quad  + \lambda \left[  W_{b^*}^{(\theta,r)}(x) \dfrac{\rho^{(\theta)}(b^*;w'_{+})}{W^{(\theta)}(b^*)}  -  \int_0^{x-{b^*}} w'_{+}(z+{b^*}) W^{(r+\theta)}(x-{b^*}-z) \diff z - \rho_{b^*}^{(\theta,r)}(x;w'_{+}) \right],
\end{split}
\]
{Now, using \eqref{r12aa} and rearranging terms we have}
\[
\begin{split}
	v'_{b^*}(x) &= \beta Z_{b^*}^{(\theta,r)}(x) - r \beta Z^{(\theta)}(b^*) \overline{W}^{(\theta + r)}(x-{b^*}) \\
	& \quad - \left( \dfrac{\beta Z^{(\theta)}(b^*) - 1 + \lambda \tilde{C}^{(\theta,r)}(b^*;w'_{+}) }{W^{(\theta)}(b^*)} \right) \left( W_{b^*}^{(\theta,r)}(x) - r W^{(\theta)}(b^*) \overline{W}^{(\theta + r)}(x-{b^*}) \right) \\
	&\quad + \lambda \E_x \left[ \int_0^{\tau_0^-(r)} e^{-\theta t} w'_{+}(U_r^{b^*}(t)) \diff t\right].
\end{split}
\]
{From expression \eqref{root_fun}} and the fact that $G(b^*) = 0$, we get
\[
\dfrac{\beta Z^{(\theta)}(b^*) - 1 + \lambda \tilde{C}^{(\theta,r)}(b^*;w'_{+}) }{W^{(\theta)}(b^*)} = \beta \theta \dfrac{Z^{(\theta)}(b^*;\Phi(\theta+r))}{Z^{(\theta)\prime}(b^*;\Phi(\theta+r))}.
\]
Hence, using \eqref{403}, we have
\begin{align}
v'_{b^*}(x) &= \beta \E_x\left[ e^{-\theta \tau_0^-(r)} \right] + \lambda \E_x \left[ \int_0^{\tau_0^-(r)} e^{-\theta t} w'_{+}(U_r^{b^*}(t)) \diff t\right]\notag \\
&= \beta - \E_x \left[ \int_0^{\tau_0^-(r)} e^{-\theta t} \left( \beta \theta - \lambda w'_{+}(U_r^{b^*}(t)) \right) \diff t \right].\label{n2}
\end{align}
From Assumption \ref{assum_w} and \eqref{n2}, it follows that $v'_{b^*}$ is non-negative and non-increasing on $(0,\infty)$. On the other hand, \eqref{n2} yields that  $v'_{b^{*}}(0)\leq\beta$ due to the fact that $w'_{+}\leq\beta$ on $(0,\infty)$. Therefore $0 \leq v'_{b^*}(x) \leq \beta $. This, and the fact that $v'_{b^*}(b^*) = 1$ {completes} the proof.

(ii) Suppose that $b^* = 0$, where necessarily $X$ has paths of bounded variation and \eqref{cond_exist_opt} holds. From \eqref{v2}, we have the following
\begin{equation}\label{constant_c_01}
C^1_0 = \dfrac{r (\beta - 1) + \theta \beta}{\theta \Phi(\theta + r)},
\end{equation}
 {Additionally, using \eqref{r12aa} we note
\begin{align}\label{resol_par_0_0}
\E_0 \left[ \int_0^{\tau_0^-(r)} e^{-\theta t} w'_{+}(U_r^{0}(t)) \diff t\right]=\tilde{C}^{(\theta,r)}(0;w'_{+})=\frac{\Phi(\theta+r)W^{(\theta)}(0+)}{\Phi(r+\theta)-rW^{(\theta)}(0+)}\int_0^{\infty}e^{-\Phi(r+\theta)y}w'_+(y)\diff y.
\end{align}
Hence, using \eqref{r13} and \eqref{resol_par_0_0} we obtain}
\begin{equation}\label{const_c_02}
\theta  {C^{(\theta,r)}(0;w)} - w(0) = \int_0^{\infty} e^{-\Phi(\theta+r)z} w'_{+}(z) \diff z.
\end{equation}
 {Then, by using \eqref{constant_c_01} and \eqref{const_c_02} in \eqref{NPV_diff} and applying \eqref{Rem_identities_simplify.0} we have
\begin{align}\label{deriv_v_0_x}
v_0'(x) &= 
  - \dfrac{r (\beta - 1) + \theta \beta}{ \Phi(\theta + r)} W^{(\theta + r)}(x) - r \overline{W}^{(\theta+r)}(x) + \beta Z^{(\theta+r)}(x) \\
& \quad + \lambda \left[ W^{(\theta + r)}(x) \int_0^{\infty} e^{-\Phi(\theta+r)z} w'_{+}(z) \diff z - \int_0^x w'_{+}(z) W^{(\theta+r)}(x-z) \diff z \right] \notag \\
&= \dfrac{r (\beta - 1) + \theta \beta}{\theta + r} \left( Z^{(\theta+r)}(x) - \dfrac{\theta+r}{\Phi(\theta+r)}W^{(\theta+r)}(x) \right) + \dfrac{r}{\theta+r}  \notag\\
&  + \lambda \left[ W^{(\theta + r)}(x) \int_0^{\infty} e^{-\Phi(\theta+r)z} w'_{+}(z) \diff z - \int_0^x w'_{+}(z) W^{(\theta+r)}(x-z) \diff z \right]  \notag \\
&= \dfrac{r (\beta - 1) + \theta \beta}{\theta + r} \E_x\left[ e^{-(\theta+r)\tau_0^-} \right] + \dfrac{r}{\theta+r}   + \lambda \E_x \left[ \int_0^{ \tau_0^-} e^{-(\theta+r) t} w'_{+}(X(t)) \diff t \right],\notag 
\end{align}
where in the last equality we used \eqref{fpt0} and \eqref{rlp}.}

Note that the mapping $x \mapsto \E_x\left[ e^{-(\theta+r)\tau_0^-} \right]$ is non-increasing, and since $w$ is concave we have that the mapping $x \mapsto \lambda \E_x \left[ \int_0^{ \tau_0^-} e^{-(\theta+r) t} w'_{+}(X(t)) \diff t \right]$ is non-increasing as well; hence  {$v_0$} is concave.

 {On the other hand, as $b^*=0$ it follows from \eqref{NPV_diff_b} that
\begin{equation}\label{der_aux_v_0_0}
	v_0'(0) = \dfrac{\theta}{\Phi(\theta+r)} \dfrac{G(0)}{1 - \E_b\left[ e^{-\theta \tau_{0}^{-}(r)} \right]} W^{(\theta)}(0)  + 1\leq 1,
\end{equation}
where the inequality follows from the fact that $G(0)\leq0$ as a result of \eqref{cond_exist_opt}. Therefore, \eqref{der_aux_v_0_0} and the concavity of $v_0$ in $(0,\infty)$ imply that}  $v_0'(x) \leq 1$ for all $x \geq 0$. Finally,  {from \eqref{deriv_v_0_x}}, we note that $v_0'(x) \rightarrow \dfrac{r + \lambda w'(+\infty)}{\theta + r} > 0$ as $x \uparrow \infty$,  hence $v_0'(x) > 0$ for all $x \geq 0$.
\end{proof}
 {By applying Lemma \ref{lem_HJB_2},} the following result is immediate.
\begin{lemma}\label{lem_HJB_3}
For $b^* \geq 0$ we have
\begin{equation}
	\max_{0 \leq l \leq x} \lbrace l + v_{b^*}(x-l) - v_{b^*}(x) \rbrace = 
	\begin{cases}
		0, & \text{if } x \in [0,b^*], \\
		 x-b^* + v_{b^*}(b^*) - v_{b^*}(x) , & \text{if } x \in (b^*,\infty).
	\end{cases}
\end{equation}
\end{lemma}
Finally, we  {provide the proof of} Theorem \ref{thm_optim_aux}.
\begin{proof}[Proof of Theorem \ref{thm_optim_aux}.] We shall show that $v_b$ satisfies the conditions from Lemma \ref{prop_HJB_aux}. As a consequence of Lemma \ref{lem_NPV_smooth}, we have that $v_{b^*}\in\mathcal{C}(\R)\cap\mathcal{C}^{1}((0,\infty))$ if $X$ has bounded variation paths ($v_{b^*}\in\mathcal{C}^{1}(\R)\cap\mathcal{C}^{2}((0,\infty))$ if $X$ has unbounded variation paths)  . Hence, it remains to prove the variational inequalities  {given in} \eqref{HJB_aux}.
	
The first item in \eqref{HJB_aux} holds with equality due to Lemmas \ref{lem_HJB_1} and \ref{lem_HJB_3}. The second item follows from Lemma \ref{lem_HJB_2}. Finally, as a consequence of Lemma \ref{lem_HJB_2}, $v_{b^*}$ is  non-decreasing and it follows  {from Assumption \ref{assump_1_aux}} that $\inf_{x \geq 0} v_{b^*}(x) =v_{b^{*}}(0) > -\infty$.
\end{proof}
\section{Optimal Strategies for Poissonian dividend problems with Regime Switching}\label{sec_optimal_MAPS}
\subsection{Iteration algorithm to compute the value function}
We will show that the net present value $V_{\pi^{0,\bb}}$, given by \eqref{va1}, of a Parisian-classical reflection strategy at levels $\bb = (b(i))_{i \in E}$ and at 0, respectively, solves a fixed point equation. 

We consider the space of functions
\begin{align*}
\B:=\{f : f(\cdot,i) \in \mathcal{C}([0,\infty))\ \text{and}\ \|f\|<\infty \ \text{for}\ i \in E\},
\end{align*}
where $\| f \| := \max_{i \in E} \sup_{x \geq 0} \dfrac{|f(x,i)|}{1 + |x|} $.

For $f: [0,\infty) \times E \rightarrow \R$ we define $\hat{f}: [0,\infty) \times E \rightarrow \R$ as
\begin{equation}\label{Expectation_Transform}
	\hat{f}(x,i) := \sum_{j \neq i} \dfrac{\lambda_{ij}}{{{\lambda}}_i} \int_{(-\infty,0)} \left[ \left( \beta(x+y) + f(0,j) \right) 1_{ \lbrace x + y < 0 \rbrace } + f(x+y,j) 1_{\lbrace x + y \geq 0 \rbrace} \right] \diff F_{ij}(y),
\end{equation}
where $F_{ij}$ denotes the distribution function of the random variable $J_{ij}$ for $i,j \in E$, and $\lambda_i := \sum_{j \neq i} \lambda_{ij}$.
\begin{remark}\label{bound_tilde_f}
	Note that for $(x,i) \in [0,\infty) \times E$ we have
	\[
	\begin{split}
		\dfrac{ |\hat{f}(x,i)| }{1 + |x|} &\leq \sum_{j \neq i} \dfrac{\lambda_{ij}}{{\lambda}_i} \left[ (\beta+|f(0,j)|)+ \dfrac{\beta\E[|J_{ij}|]}{1+|x|} + \|f\|\int_{(-\infty,0)} \dfrac{1+|x+y|}{ 1+|x| }1_{\{x+y\geq0\}} \diff F_{ij}(y) \right] \\
		& \leq \sum_{j \neq i} \dfrac{\lambda_{ij}}{{\lambda}_i} \left[(\beta+|f(0,j)|)+ \beta\E[|J_{ij}|]+\|f\|\right].
	\end{split}
	\]
	Hence, if $f \in \B$, we have that $\hat{f} \in \B$ as well.
\end{remark}
Given $\mathbf{b} = (b(i))_{i \in E}\in\mathcal{E}$, where $\mathcal{E}$ denotes the space of functions from $E$ to $[0,\infty)$, we define the following operator acting on $\B$
{
\begin{align}
	T_{\bb} f(x,i) :  
	= \E^i_x\Bigg[ \int_{[0,\infty)}& e^{-\theta(i)t} \diff L_i^{0,b(i)}(t) 
	\\&- \beta  \int_{[0,\infty)} e^{-\theta(i)t} \diff R_i^{0,b(i)}(t) 
	+ \lambda_{i} \int_0^{\infty}e^{-\theta(i)t} \hat{f}(U_{r,i}^{0,b(i)}(t),i)\diff t \Bigg].\notag
\label{Recursion_Operator}
\end{align}}
where $\theta(i)=q(i)+\lambda_i$ and $U_{r,i}^{0,b(i)}$ denotes the process with Parisian-classical reflection at the thresholds $b(i)$ and 0, respectively, driven by $X^i$; and $L_i^{0,b(i)}, R_i^{0,b(i)}$ are the cumulative dividend payments and capital injections, respectively.

\begin{proposition}\label{Prop_FixedPoint} For 
	 $\bb\in\mathcal{E}$, and $(x,i) \in [0,\infty) \times E$ we have
	\[
	V_{\pi^{0,\bb}}(x,i) = T_{\bb} V_{\pi^{0,\bb}}(x,i). 
	\]
\end{proposition}
\begin{proof}
	We denote by $\zeta$ to the epoch of the first regime switch. Hence, by an application of the strong Markov property, we obtain
	\begin{align*}
		 V_{\pi^{0,\bb}} (x,i) &=  {\Em}_{(x,i)} \left[ \int_{[0, \infty)} e^{-I(t)}\diff L_r^{0,\bb}(t) -  \beta\int_{[0, \infty)} e^{-I(t)} \diff R_r^{0,\bb}(t)\right]\\
		&=  {\Em}_{(x, i)} \bigg{[}\int_{[0, \zeta)} e^{- q(i) t}\diff L_{i}^{0,b(i)}(t)
		-\beta \int_{[0, \zeta)} e^{-q(i) t}\diff R_{i}^{0,b(i)}(t)
		+e^{-q (i) \zeta }\bigg[\beta (U_{r,i}^{0,b(i)}(\zeta-)+J_{iH(\zeta)})\notag\\
		& \quad+  {\Em}_{(0, H({\zeta}))} \bigg{[}\int_{[0, \infty)} e^{-  I(t)}\diff L_r^{0,\bb}(t)
		-\beta\int_{[0, \infty)} e^{- I(t)} \diff R_r^{0,\bb}(t) \bigg{]}\bigg]1_{\{U_{r,i}^{0,b(i)}(\zeta-)< -J_{iH(\xi))}\}}\notag\\
		& \quad+ 
		e^{-q(i)\zeta}  {\Em}_{(U_{r}^{0,\bb}(\zeta), H({\zeta}))} \bigg[\int_{[0, \infty)} e^{-  I(t)}\diff L_r^{0,\bb}(t)-\beta\int_{[0, \infty)} e^{- I(t)} \diff R_r^{0,\bb}(t) \bigg{]}1_{\{U_{r,i}^{0,b(i)}(\zeta-)\geq -J_{iH(\xi)}\}}\bigg{]}
					\end{align*}
	\begin{align}\label{fix-point_1}
		& =  {\Em}_{(x, i)} \bigg{[}\int_{[0, \zeta)} e^{- q(i) t}\diff L_{i}^{0,b(i)}(t)
		-\beta\int_{[0, \zeta)} e^{-q(i) t}dR_{i}^{0,b(i)}(t)\notag\\&\hspace{1.7cm}+e^{-q(i)\zeta}V_{\pi^{0,\bb}}(U_{r}^{0,\bb}(\zeta), H(\zeta))1_{\{U_{r,i}^{0,b(i)}(\zeta-)\geq -J_{iH(\xi)}\}} \notag\\
		&\hspace{1.7cm}+e^{-q (i) \zeta }\Big(\beta (U_{r,i}^{0,b(i)}(\zeta-)+J_{iH(\zeta)})+V_{\pi^{0,\bb}}(0, H(\zeta))\Big)1_{\{U_{r,i}^{0,b(i)}(\zeta-)< -J_{iH(\xi)}\}}\bigg{].}
	\end{align}
	By conditioning on the state of the Markov chain $H$ at the first regime switching time $\zeta$ and the random variable $J_{ij}$, describing the jump when $H$ makes a transition from the state $i$ to $j$, we get
	\begin{align}\label{fix-point_2}
		{\Em}_{(x, i)} \bigg[&\int_{[0, \zeta)} e^{- q(i) t}\diff L^{0,b(i)}_{i}(t)
		-\beta\int_{[0, \zeta)} e^{-q(i) t}dR^{0,b(i)}_{i}(t)+e^{-q(i)\zeta}V_{\pi^{0,\bb}}(U^{0,\bb}_{r}(\zeta), H(\zeta))1_{\{U_{r,i}^{0,b(i)}(\zeta-)\geq -J_{iH(\xi)}\}} \\
		& + e^{-q (i) \zeta }\Big(\beta (U_{r,i}^{0,b(i)}(\zeta-)+J_{iH(\zeta)})+V_{\pi^{0,\bb}}(0, H(\zeta))\Big)1_{\{U_{r,i}^{0,b(i)}(\zeta-)< -J_{iH(\zeta)}\}}\bigg]\notag\\
		& = \sum_{j\in E,j\not=i}\dfrac{\lambda_{ij}}{{\lambda}_i}  {\Em}_{(x, i)} \bigg{[}\int_{[0, \zeta)} e^{- q(i) t}\diff L^{0,b(i)}_{i}(t)
		-\beta\int_{[0, \zeta)} e^{-q(i) t}dR^{0,b(i)}_{i}(t)\notag\\
		& \hspace{3cm}+e^{-q(i)\zeta}\Big[V_{\pi^{0,\bb}}(U_{r,i}^{0,b(i)}(\zeta-)+J_{ij}, j)1_{\{U_{r,i}^{0,b(i)}(\zeta-)\geq -J_{ij}\}}\notag\\
		&\hspace{3cm}+\Big(\beta(U_{r,i}^{0,b(i)}(\zeta-)+J_{ij})+V_{\pi^{b(i)}}(0,j)\Big)1_{\{U_{r,i}^{0,b(i)}(\zeta-)<- J_{ij}\}}\Big]\bigg|H_{\zeta}=j\bigg{]}\notag\\
		& =  {\Em}_{(x, i)} \bigg{[}\int_{[0, \zeta)} e^{- q(i) t}\diff L^{0,b(i)}_{i}(t)
		-\beta\int_{[0, \zeta)} e^{-q(i) t} \diff R^{0,b(i)}_{i}(t)
		+e^{-q(i)\zeta}{\hat{V}}_{\pi^{0,\bb}}(U_{r,i}^{0,b(i)}(\zeta-), i)\bigg{]}\notag\\
		&=T_{\bb}V_{\pi^{0,\bb}}(x, i).\notag
	\end{align}
	In the last equality we used that $\zeta$ is an exponential random variable with rate ${\lambda}_i$, independent of the processes $L^{0,\bb}$, $R^{0,\bb}$, and $U^{0,\bb}$.
\end{proof}
For the next result we denote, for any $f,g\in\mathcal{B}$, 
$
\|f-g\|_{\infty}:=\max_{i \in E} \sup_{x \geq 0}|f(x,i)-g(x,i)|.
$
\begin{lemma}\label{norm_contraction}
If $\bb\in\mathcal{E}$ and  $f,g\in\mathcal{B}$  satisfy $\|f-g\|_{\infty} <\infty$, then
	\begin{equation*}
	\|T_{\bb}f-T_{\bb}g\|_{\infty}<{K}\|f-g\|_{\infty},
	 \end{equation*}
 where ${K}:=\max_{i\in E}\{{\lambda}_{i}/\theta(i)\}<1$.
\end{lemma}
\begin{proof}
	Since
	\begin{align*}
		T_{\bb} f(x,i) 
	= \E^i_x&\left[ \int_{[0,\infty)} e^{-\theta(i)t} \diff L_i^{0,b(i)}(t) - \beta  \int_{[0,\infty)} e^{-\theta(i)t} \diff R_i^{0,b(i)}(t) \right]\\
	&\qquad \qquad 
	+ \E^i_x\left[e^{-q(i)\zeta} \hat{f}(U_{r,i}^{0,b(i)}(\zeta),i) \right],
	\end{align*}
	{for each $(x,i)\in[0,\infty)\times E$,}
	\begin{align*}
	| T_{\bb} f(x,i) - T_{\bb} g(x,i)|&\leq \sum_{j \neq i} \dfrac{\lambda_{ij}}{{\lambda}_i}  {\Em}_{(x,i)}\left[ e^{-q(i)\zeta} \int_{(-\infty,-U_{r,i}^{0,b(i)}(\zeta-))} |f(0,j) - g(0,j)| \right. \diff F_{ij}(y) \notag\\
	& + \left. e^{-q(i)\zeta }\int_{[-U_{r,i}^{0,b(i)}(\zeta-),0]} | f(U_{r,i}^{0,b(i)}(\zeta-)+y,j) - g(U_{r,i}^{0,b(i)}(\zeta-)+y,j) | \diff F_{ij}(y) \right]\notag\\
	&\leq\|f-g\|_{\infty} \sum_{j \neq i} \dfrac{\lambda_{ij}}{{\lambda}_i}  {\Em}_{(0,i)}[ e^{-q(i)\zeta}]\leq {K}\|f-g\|_{\infty}.\notag
	\end{align*}
\end{proof}

\subsection{Verification of barrier strategies.}
We define the space of functions
\begin{align*}
	\mathcal{D}:= \lbrace f \in \B : \hat{f}(\cdot,i) &\text{ is concave and satisfies that $\hat{f}'_+(0,i)\leq \beta$ }\text{and $\hat{f}'_+(\infty,i)\in[0,1]$ for $i \in E$} \rbrace
\end{align*}
\begin{proposition}\label{gammainD}
	Consider $f \in \B$ such that 
	it is concave, it is nondecreasing, and satisfies that $f'_{+}(\cdot,i)\leq \beta$ and $f'_{+}(\infty,i)\leq 1$ for all $i \in E$. Then $f \in \DD$. 
\end{proposition}
\begin{proof}
	Using \eqref{Expectation_Transform} and integration by parts, we obtain
	By \eqref{Expectation_Transform} and the Dominated Convergence Theorem, we have
	\begin{align*}
	\hat{f}^\prime_+&(x,i)=\lim_{\varepsilon\downarrow 0}\frac{\hat{f}(x+\varepsilon,i)-\hat{f}(x,i)}{\varepsilon}\\
	&=\sum_{j \neq i} \dfrac{\lambda_{ij}}{{{\lambda}}_i} \int_{(-\infty,0)} \lim_{\varepsilon\downarrow0}\frac{1}{\varepsilon}\bigg{[}
	\left\{ \left( \beta(x+\varepsilon+y) + f(0,j) \right) 1_{ \lbrace x +\varepsilon+ y < 0 \rbrace } + f(x+\varepsilon+y,j) 1_{\lbrace x+\varepsilon + y \geq 0 \rbrace} \right\} \\
	&- \left\{ \left( \beta(x+y) + f(0,j) \right) 1_{ \lbrace x + y < 0 \rbrace } + f(x+y,j) 1_{\lbrace x + y \geq 0 \rbrace} \right\}\bigg{]} \diff F_{ij}(y).
	\end{align*}
	We fix $x\geq 0$ and $y\leq0$, then we have $1_{\lbrace x+\varepsilon + y \geq 0 \rbrace}=1_{\lbrace x+ y \geq 0 \rbrace}$ and $1_{\lbrace x+\varepsilon + y < 0 \rbrace}=1_{\lbrace x+ y < 0 \rbrace}$ for small enough $\varepsilon>0$. Thus, we have
		\begin{align*}
	\hat{f}^\prime_+(x,i)=\sum_{j \neq i} \dfrac{\lambda_{ij}}{{{\lambda}}_i} \int_{(-\infty,0)} \left[ \beta 1_{ \lbrace x + y < 0 \rbrace } + f^\prime_+(x+y,j) 1_{\lbrace x + y \geq 0 \rbrace} \right] \diff F_{ij}(y)
	\end{align*}
	and $\hat{f}$ is continuous on $[0, \infty)$. 
	In addition, $\hat{f}^\prime_+$ is non-increasing since $f$ is concave and its right derivative is no more than $\beta$. 
	Thus, by Theorem 6.4 in \cite{HL2001}, the function $\hat{f}$ is concave. We have $\hat{f}^\prime_+\leq \beta$ since $f^\prime_+\leq \beta$. Additionally, by dominated convergence
	\begin{align*}
		\hat{f}'(+ \infty,i)  = \sum_{j \neq i} \dfrac{\lambda_{ij}}{{\lambda}_i} f'(+\infty,j)\leq 1.
	\end{align*}
	Therefore, we obtain that $f\in\DD$.
\end{proof}
For $f \in \DD$ and $(x,i) \in [0,\infty)\times E$ we define
\begin{align}\label{Supremum_Operator}
	\Gamma f(x,i) :&= \sup_{\pi \in \mathcal{A}}  {\Em}_{(x,i)}\left[ \int_{[0,\zeta)} e^{-q(i)t} \diff L^{\pi}_{r}(t) - \beta \int_{[0,\zeta)} e^{-q(i)t} \diff R^{\pi}_{r}(t) + e^{-q(i)\zeta} \hat{f}(U^{\pi}_r(\zeta -),i) \right]\\
	&=\sup_{\pi \in \mathcal{A}} \E_{x}^i\left[ \int_{[0,\infty)} e^{-\theta(i)t} \diff L^{\pi}_{r}(t) - \beta \int_{[0,\infty)} e^{-\theta(i)t} \diff R^{\pi}_{r}(t) + {\lambda}_i \int_0^{\infty}e^{-\theta(i)t} \hat{f}(U_r^{\pi}(t),i) dt\right].\notag
\end{align}
\begin{remark}\label{Remark_Optimal_Barrier}
	Since ${f} \in \DD$, then Theorem \ref{thm_optim_aux} guarantees that there exists $b^f(i)$ such that the supremum in the second equality \eqref{Supremum_Operator} is attained by a  periodic-classical reflection strategy at the barrier $b^f(i)$ and at 0, respectively, for each $i \in E$. Hence, by taking $\bb^f = (b^f(i))_{i\in E}$ we get $\Gamma f = T_{\bb^f} f$. From the verification results of Section \ref{verification_aux} it follows, for $f\in\mathcal{D}$ and $i\in E$, that $\Gamma f(\cdot,i) \in \mathcal{C}^1$, it has linear growth, it is concave, $(\Gamma f)'(0+,i)\leq \beta$, and $(\Gamma f)'(\infty,i)\in [0,1] $. Therefore, by Proposition \ref{gammainD}, we have that $\Gamma f \in \DD$.	
\end{remark}
The following result follows the same line of reasoning as in Proposition 5.4 in Noba et al. \cite{NobPerYu}.
\begin{proposition}\label{Prop_Supremum_FixedPoint}
	Let $v_0^-, v_0^+ \in \DD$, and for $n \geq 1$ we define $v_n^- := \Gamma v^-_{n-1}$ and $v^+_n := \Gamma v^+_{n-1}$. If $v^-_0 \leq V \leq v^+_0$ and $\| v^+_0-v^-_0 \|_{\infty}<\infty $, then we have $v^-_n \leq V \leq v^+_n$ for all $n \geq 1$. Moreover, we have
	\[
	V(x,i) = \lim_{n \rightarrow \infty} v^-_n(x,i) = \lim_{n \rightarrow \infty} v^+_n(x,i),
	\]
	where the convergence is in the {$\|\cdot\|_{\infty}$-norm (and thus in the $\|\cdot\|$-norm)}. In particular, $V \in \DD$.
\end{proposition}
\begin{proof}
	First, from the definition of $\Gamma$ we have that if $v^-_{n-1} \leq V \leq v^+_{n-1}$, then
	\begin{equation}\label{ine_V}
		v^-_n=\Gamma v_{n-1}^- \leq \Gamma V \leq \Gamma v_{n-1}^+\leq v^+_n.
	\end{equation}
	Proceeding by induction, we obtain the first claim.
	
	Following Remark \ref{Remark_Optimal_Barrier}, for any $f \in \DD$ there exists  $\bb^f\in\mathcal{E}$, such that $\Gamma f = \sup_{\bb\in\mathcal{E}} T_{\bb} f = T_{\bb^f} f$. Hence, using  {Lemma} \ref{norm_contraction}, we obtain  for $f,g \in \DD$ such that $\|f-g\|_{\infty}<\infty$ 
%
	 		\[
	 		\| \Gamma f - \Gamma g \|_{\infty}=\|T_{\bb^f}f-T_{\bb^g}g \|\leq \sup_{b\in\mathcal{E}} \| T_b f - T_b g \|_{\infty} \leq K\| f-g \|_{\infty}.
	 		\]

	Then, by an iteration of the previous identity and the definition of $v_n^+$ and $v_n^-$, we obtain  
		\[
		\| v_n^+-v_n^-\|_{\infty}  \leq   K^n \| v_0^+-v_0^- \|_{\infty}\quad\text{for}\ n\in\mathbb{N}.
		\]
		This implies that $\|v^{+}_{n}-v^{-}_{n}\|_{\infty}\underset{n\rightarrow\infty}{\longrightarrow}0$ given that  $K<1$. Therefore, using \eqref{ine_V}, we have
	\[
	\lim_{n\to\infty}v_n^+=\lim_{n\to\infty}v_n^-=\Gamma V=V.
	\]
	
	Following Remark \ref{Remark_Optimal_Barrier} we have that the functions $v_n^+, \, v^-_n$ belong to $\DD$ for all $n \geq 1$. On the other hand, using dominated convergence we obtain for $(x,i) \in [0,\infty) \times E$
	\begin{align*}
		\hat{V}(x,i) &= \lim_{n \rightarrow \infty}  \sum_{j \neq i} \dfrac{\lambda_{ij}}{{\lambda}_i} \int_{(-\infty,0)} \left[ \left( \beta(x+y) + \hat{v}_{n}^{\pm}(0,j) \right) 1_{ \lbrace x + y < 0 \rbrace } +\hat{v}_{n}^{\pm}(x+y,j) 1_{\lbrace x + y \geq 0 \rbrace} \right] \diff F_{ij}(y)\notag\\
		&= \lim_{n \rightarrow \infty}  \hat{v}_{n}^{\pm}(x,i).
	\end{align*}
	Hence, using the fact that the functions $v^{\pm}_n \in \DD$ we obtain that $V \in \DD$ as well.
\end{proof}
	{We will provide two auxiliary results that will be used in the proof of Theorem \ref{thm_optimal_barrier}. The first result guarantees the existence of functions $v_0^-, v_0^+ \in \DD$ that satisfy the conditions of Proposition \ref{Prop_Supremum_FixedPoint}, and its proof is deferred to Appendix \ref{proof_lemma_v}.
	\begin{lemma}\label{proof_lemma_vv}
		There exist $V_{-}, V_{+} \in \DD$ such that $\|V_--V_+\|_{\infty}<\infty$ and 
			\[
		V_{-}(x,i) \leq V(x,i) \leq V_{+}(x,i), \quad (x,i) \in [0,\infty) \times E.
		\]
	\end{lemma}
	We now provide the second auxiliary result and defer its proof to Appendix \ref{proof_inf_norm_finite}.
	\begin{lemma}\label{inf_norm_finite}
		For any $\bb\in\mathcal{E}$ we have that $\|V-{V_{\pi^{0,\bb}}}\|_{\infty}<\infty$.
	\end{lemma}}
{Using Lemma \ref{proof_lemma_vv} together with Proposition \ref{Prop_Supremum_FixedPoint}}, we provide an iterative construction of the value function as follows: Initialize $n=0$ and let $v = v_0$ for some $v_0 \in \DD$, then we proceed as follows:
\begin{enumerate}
	\item Compute $\bb^v = (b^v(i))_{i \in E}$ as in Remark \ref{Remark_Optimal_Barrier};
	\item Set $T_{\bb^v}v \rightarrow v, \, n \rightarrow n+1$, and $v \rightarrow v_n$, and return to step 1.
\end{enumerate}
\subsection{Proof of Theorem \ref{thm_optimal_barrier}}
Due to Proposition \ref{Prop_Supremum_FixedPoint}, we have $V \in \DD$. Hence, Proposition \ref{MAP_DPP} together with Remark \ref{Remark_Optimal_Barrier} imply that there exists $\bb^* = (b^*(i))_{i\in E}$ such that
\[
V(x,i) = \Gamma V(x,i) = T_{\bb^*}V(x,i), \quad (x,i) \in [0,\infty) \times E.
\]
This further yields that
\begin{align}
	V(x, i)=\lim_{n\uparrow \infty}T_{\bb^\ast}^n V(x, i), \qquad\text{for $(x,i)\in[0,\infty)\times E$.} 
\end{align}
Finally, an application of Proposition \ref{Prop_FixedPoint} together with {Lemma} \ref{norm_contraction} implies that
\begin{align}\label{fix_point_final}
	\|V-{V_{\pi^{0,\bb^*}}}\|_{\infty}\leq\|T_{\bb^*}^nV-T_{\bb^*}^n{V_{\pi^{0,\bb^*}}}\|_{\infty}\leq  
	K^{n} \| V-{V_{\pi^{0,\bb^*}}} \|_{\infty},\qquad\text{$n\in\mathbb{N}$.}
\end{align}
{By Lemma \ref{inf_norm_finite} we have that $\|V-{V_{\pi^{0,\bb^*}}}\|_{\infty}<\infty$, this together with \eqref{fix_point_final} implies that} $V(x, i)={V_{\pi^{0,\bb^*}}}(x,i)$ for $(x,i)\in[0,\infty)\times E$. 
\section*{Acknowledgement}
K. Noba was supported by JSPS KAKENHI Grant Number 21K13807 and JSPS Open Partnership Joint Research Projects Grant Number JPJSBP120209921.

\appendix

\section{{Proof of Propositions \ref{resol_par} and \ref{resol_par_ruin}}}\label{resol_par_app}

\subsection{Fluctuation identities}\label{App_Fluctuation}
In this section we provide a review of some fluctuaction identities for the processes $X$ and $Y$ that will be used in the proofs of Theorems \ref{resol_par} and \ref{resol_par_ruin}.
\subsubsection{Identities for the process $X$}
Let us define for $b\in\R$
\begin{align}
\tau_b^-:=\inf\{t>0:X(t)<b\},\qquad \tau_b^+:=\inf\{t>0:X(t)>b\}.
\end{align}
Then, by Theorem 8.1 in \cite{Kyp}
\begin{align}
\E_x\left[e^{-q\tau_0^-};\tau_0^-<\infty\right]&=Z^{(q)}(x)-\dfrac{q}{\Phi(q)}W^{(q)}(x),\label{fpt0}\\ E_x\left[e^{-q\tau_b^+};\tau_b^+<\tau_0^-\right]&=\dfrac{W^{(q)}(x)}{W^{(q)}(b)}\label{fpt},
\end{align}
On the other hand, by Theorem 2.7(i) in \cite{KKR}, it follows that for any measurable function $h:\R\mapsto\R$,
\begin{align}\label{eq2}
\E_x\left[\int_0^{\tau_b^+\wedge\tau_0^-}e^{-qt}h(X(t))dt\right]&=\int_{0}^{b}h(y)\left\{\dfrac{W^{(q)}(x)W^{(q)}(b-y)}{W^{(q)}(b)}-W^{(q)}(x-y)\right\}\diff y\\
&=\dfrac{W^{(q)}(x)}{W^{(q)}(b)}\rho^{(q)}_{b}(b;h)-\rho^{(q)}_{b}(x;h), \qquad x\in[0,b],\notag
\end{align}
where $\rho_b^{(q)}$ is defined in \eqref{fun_rho}. 

Additionally, by Theorem 2.7(iii) in \cite{KKR}, for any bounded measurable function $h:\R\mapsto\R$ with compact support,
\begin{align}
\E_x\left[\int_0^{\tau_0^-}e^{-qt}h(X(t))dt\right]&=\int_0^{\infty}h(y)\left\{e^{-\Phi(q)y}W^{(q)}(x)-W^{(q)}(x-y)\right\}\diff y.\label{rlp} \qquad x \geq0.
\end{align}
Finally, by an application of Lemma 2.1 in \cite{lrz} we obtain the following identities for $x > b$
\begin{align}
\E_x\left[e^{-(q+r)\tau_b^-} W^{(q)}(X(\tau_b^-));\tau_b^-<\infty\right] &= W^{(q,r)}_b(x)-rW^{(q+r)}(x-b)\int_b^{\infty}e^{-\Phi(q+r)(y-b)}W^{(q)}(y)\diff y,\label{lrz} \\
\E_x\left[e^{-(q+r)\tau_b^-} Z^{(q)}(X(\tau_b^-));\tau_b^-<\infty\right] &= Z^{(q,r)}_b(x)-rW^{(q+r)}(x-b)\int_b^{\infty}e^{-\Phi(q+r)(y-b)}Z^{(q)}(y)\diff y,  \notag\\
\E_x\left[e^{-(q+r)\tau_b^-} \rho^{(q)}_{b}(X(\tau_b^-);h);\tau_b^-<\infty\right] &= \rho^{(q,r)}_b(x;h)-rW^{(q+r)}(x-b)\int_b^{\infty}e^{-\Phi(q+r)(y-b)}\rho^{(q)}_{b}(y;h)\diff y\notag. 
\end{align}

\subsubsection{Identities for the process $Y$}
We know provide a review of some fluctuation identities for the L\'evy process reflected at the lower boundary $0$, given by $Y_t:=X_t+\sup_{0\leq s\leq t}(-X_s)\vee0$ for $t\geq0$. For any $b\geq0$, let $\kappa_b^+:=\inf\{t>0: Y_t>b\}$, then by Theorem 2.8(i) in \cite{KKR}
\begin{align}\label{fupt_r}
\E_x\left[e^{-q\kappa_b^+}\right]=\dfrac{Z^{(q)}(x)}{Z^{(q)}(b)}.
\end{align}
Additionally, by Theorem 2.8(iii) of \cite{KKR} we have that for any bounded measurable function $h:\R\mapsto\R$ with compact support 
\begin{align}\label{rrb}
\E_x\left[\int_0^{\kappa_b^+}e^{-qt}h(Y(t)) \diff t \right]=\dfrac{Z^{(q)}(x)}{Z^{(q)}(b)}\rho^{(q)}_{b}(b;h)-\rho^{(q)}_{b}(x;h),\qquad x\in[0,b]
\end{align}
Additionally, by taking $b\to\infty$ in \eqref{rrb} we obtain
\begin{align}\label{rrbi}
\E_x\left[\int_0^{\infty}e^{-qt}h(Y(t)) \diff t \right]=Z^{(q)}(x)\dfrac{\Phi(q)}{q}\int_0^{\infty}e^{-\Phi(q)y}h(y)dy-\int_0^{x}W^{(q)}(x-y)h(y)dy.
\end{align}


\subsection{Proof of Proposition \ref{resol_par}}\label{App_Proof_1}

Consider  $g^{(q)}(\cdot\,;h)$ as in \eqref{r12}.

(i) Using the strong Markov property and the absence of positive jumps, we obtain for $x<b$
\begin{align}
g^{(q)}(x;h)=\E_x\left[\int_0^{\kappa_b^+}e^{-qt}h(Y(t))\diff t\right]+\E_x\left[e^{-q\kappa_b^+}\right]g^{(q)}(b;h).
\end{align}
Then, by \eqref{fupt_r} and \eqref{rrb} we obtain that
\begin{align}\label{r2}
g^{(q)}(x;h)=\dfrac{Z^{(q)}(x)}{Z^{(q)}(b)}\left(\rho^{(q)}_{b}(b;h)+g^{(q)}(b;h)\right)-\rho^{(q)}_{b}(x;h).
\end{align}
On the other hand, for $x > b$, note that $\lbrace U_r^{(0,b)}(t) : t \leq T(1)\wedge \tau_b^- \rbrace$ started at $x$ is equal in law to $\lbrace X(t) : t \leq T(1)\wedge \tau_b^- \rbrace$ started at $x$, as well. By using this, and the strong Markov property we obtain
\begin{align}\label{r7}
g^{(q)}(x;h)&=\E_x\left[e^{-qT(1)};T(1)<\tau_b^-\right]g^{(q)}(b;h)\\
&\quad+\E_x\left[e^{-q\tau_b^-}g^{(q)}(X(\tau_b^-);h);\tau_b^-<T(1)\right]+\E_x\left[\int_0^{\tau_b^-\wedge T(1)}e^{-qt}h(X(t))\diff t\right].\notag
\end{align}
Now by taking expectation w.r.t. $T(1)$ and by \eqref{fpt0}, we get
\begin{align}\label{r1}
\E_x\left[e^{-qT(1)};T(1)<\tau_b^-\right]&=\dfrac{r}{r+q}\E_x\left[1-e^{-(q+r)\tau_b^-}\right]\\
&=\dfrac{r}{r+q}\left(1-Z^{(q+r)}(x-b)+\dfrac{q+r}{\Phi(q+r)}W^{(q+r)}(x-b)\right).\notag
\end{align}
Also, by an application of the spatial homogeneity of $X$ together with \eqref{rlp}
\begin{align}\label{r6}
\E_x\left[\int_0^{\tau_b^-\wedge T(1)}e^{-qt}h(X(t))\diff t\right]&=\E_x\left[r\int_0^{\infty}e^{-rs}\int_0^{\tau_b^-}e^{-qt}h(X(t))1_{\{t<s\}}\diff t\diff s\right]\\
&=\E_x\left[\int_0^{\tau_b^-}e^{-(q+r)s}h(X(s))\diff s\right]\notag\\
&=\E_{x-b}\left[\int_0^{\tau_0^-}e^{-(q+r)s}h(X(s)+b)\diff s\right]\notag\\
&= \int_0^{\infty}h(y+b)\left\{e^{-\Phi(q+r)y}W^{(q+r)}(x-b)-W^{(q+r)}(x-b-y)\right\}\diff y.\notag
\end{align}
Finally, by \eqref{r2} we have
\begin{align}\label{r3}
\E_x\left[e^{-q\tau_b^-}g^{(q)}(X(\tau_b^-);h);\tau_b^-<T(1)\right]&=\E_x\left[e^{-(q+r)\tau_b^-}g^{(q)}(X(\tau_b^-);h);\tau_b^-<\infty\right]\\
&=\dfrac{(\rho^{(q)}_{b}(b;h)+g^{(q)}(b;h))}{Z^{(q)}(b)}\E_x\left[e^{-(q+r)\tau_b^-}Z^{(q)}(X(\tau_b^-));\tau_b^-<\infty\right]\notag\\
&\quad-\E_x\left[e^{-(q+r)\tau_b^-}\rho^{(q)}_{b}(X(\tau_{b}^{-});h);\tau_b^-<\infty\right]\notag.
\end{align}
We can further expand \eqref{r3} using \eqref{lrz}, which yields 
\begin{align}\label{r9}
\E_x&\left[e^{-q\tau_b^-}g^{(q)}(X(\tau_b^-);h);\tau_b^-<T(1)\right]=\E_x\left[e^{-(q+r)\tau_b^-}g^{(q)}(X(\tau_b^-);h);\tau_b^-<\infty\right]\\
&=\dfrac{(\rho^{(q)}_{b}(b;h)+g^{(q)}(b;h))}{Z^{(q)}(b)}\left(Z^{(q,r)}_b(x)-rW^{(q+r)}(x-b) \int_b^{\infty}  e^{-\Phi(q+r)(y-b) }Z^{(q)}(y)\diff y\right)\notag\\
& \quad -\rho^{(q,r)}_b(x)+rW^{(q+r)}(x-b) \int_b^{\infty} e^{-\Phi(q+r) (y-b) }\rho^{(q)}_{b}(y;h)\diff y.\notag
\end{align}

Thus, by combining \eqref{r1}, \eqref{r6}, and \eqref{r9} in \eqref{r7} we obtain for $x > b$
\begin{align}\label{r10}
g^{(q)}(x;h)&= \dfrac{(g^{(q)}(b;h)+\rho^{(q)}_b(b;h))}{Z^{(q)}(b)}Z^{(q,r)}_b(x)-\rho^{(q,r)}_b(x)-rg^{(q)}(b;h)\overline{W}^{(q+r)}(x-b)\\
&-\int_0^{x-b}h(y+b)W^{(q+r)}(x-b-y)\diff y+W^{(q+r)}(x-b)\bigg(\dfrac{rg^{(q)}(b;h)}{\Phi(q+r)}\notag\\
&-\bigg(\dfrac{(\rho^{(q)}_b(b;h)+g^{(q)}(b;h))}{Z^{(q)}(b)}\bigg)r\int_b^{\infty}e^{-\Phi(q+r)(y-b) }Z^{(q)}(y)\diff y\notag\\
&+r\int_b^{\infty}e^{-\Phi(q+r) (y-b) }\rho^{(q)}_b(y;h)\diff y+\int_0^{\infty}h(y+b)e^{-\Phi(q+r)y}\diff y\bigg).\notag
\end{align}
\par (ii) On the other hand, by an application of the strong Markov property we obtain that
\begin{align}\label{g_b_1}
g^{(q)}(b;h)=\delta_1+g^{(q)}(b;h)\delta_2+\delta_3,
\end{align}
where
\begin{align*}
\delta_1:&=\E_b\left[\int_0^{T(1)}e^{-qt}h(Y(t)) \diff t\right],\notag\\
\delta_2:&=\E_b\left[e^{-qT(1)}1_{\{Y(t)\geq b\}}\right],\notag\\
\delta_3:&=\E_b\left[e^{-qT(1)}g^{(q)}(Y(T(1));h)1_{\{Y(T(1))<b\}}\right].
\end{align*}
Now, using identity \eqref{rrbi} we can write
\begin{align}\label{delta_1}
\delta_1&=\E_b\left[\int_0^{\infty}e^{-(q+r)t}h(Y(t)) \diff t\right]\\
&=Z^{(q+r)}(b)\dfrac{\Phi(q+r)}{r+q}\int_0^{\infty}e^{-\Phi(q+r)y}h(y)dy-\int_0^bW^{(q+r)}(b-y)h(y)\diff y,\notag
\end{align}
and
\begin{align}\label{delta_2}
\delta_2&=r\E_b\left[\int_0^{\infty}e^{-(q+r)t}1_{\{Y(t)\geq b\}} \diff t\right]\\
&=rZ^{(q+r)}(b)\dfrac{\Phi(q+r)}{r+q}\int_b^{\infty}e^{-\Phi(q+r)y} \diff y=\dfrac{r}{r+q}Z^{(q+r)}(b)e^{-\Phi(q+r)b}.\notag
\end{align}
For the last term in \eqref{g_b_1}, we once again use \eqref{rrbi} together with \eqref{r2} to obtain
\begin{align}\label{delta_3}
\delta_3&=r \E_b\left[\int_0^{\infty}e^{-(q+r)t}g^{(q)}(Y(t);h)1_{\{Y(t)<b\}} \diff t\right]\\
&=rZ^{(q+r)}(b)\dfrac{\Phi(q+r)}{r+q}\int_0^{b}e^{-\Phi(q+r)y}g^{(q)}(y;h) \diff y-r\int_0^bW^{(q+r)}(b-y)g^{(q)}(y;h) \diff y\notag\\
&=rZ^{(q+r)}(b)\dfrac{\Phi(q+r)}{r+q}\left[\dfrac{\rho^{(q)}_{b}(b;h)+g^{(q)}(b;h)}{Z^{(q)}(b)}\int_0^{b}e^{-\Phi(q+r)y}Z^{(q)}(y) \diff y-\int_0^{b}e^{-\Phi(q+r)y}\rho_b^{(q)}(y;h) \diff y\right]\notag\\
&\quad-r\dfrac{\rho^{(q)}_{b}(b;h)+g^{(q)}(b;h)}{Z^{(q)}(b)}\int_0^{b}W^{(q+r)}(b-y)Z^{(q)}(y)dy+r\int_0^{b}W^{(q+r)}(b-y)\rho_b^{(q)}(y;h) \diff y.\notag
\end{align}
Using integration by parts together with \eqref{scale_function_laplace}, we get
\begin{align*}
\int_0^{\infty}e^{-\Phi(q+r)y}Z^{(q)}(y) \diff y=\dfrac{1}{\Phi(q+r)}+\dfrac{q}{\Phi(q+r)}\int_0^{\infty}e^{-\Phi(q+r)y}W^{(q)}(y) \diff y=\dfrac{r+q}{r\Phi(q+r)}.
\end{align*}
Hence,
\begin{align}\label{g_b_2}
r\int_0^{b}e^{-\Phi(q+r)y}Z^{(q)}(y) \diff y=\dfrac{r+q}{\Phi(q+r)}-r\int_b^{\infty}e^{-\Phi(q+r)y}Z^{(q)}(y) \diff y.
\end{align}
Now, using identity \eqref{eq1},
\begin{align}\label{g_b_3}
r\int_0^{b}W^{(q+r)}(b-y)Z^{(q)}(y)dy=Z^{(q+r)}(b)-Z^{(q)}(b).
\end{align}
Additionally, using Fubini's theorem together with \eqref{eq1}, gives 
\begin{align}\label{g_b_4}
r\int_0^bW^{(q+r)}(b-y)\rho_b^{(q)}(y;h) \diff y&=r\int_0^b W^{(q+r)}(b-y)\int_0^{b}W^{(q)}(y-u)h(u)\diff u \diff y\\
&=r\int_0^bh(u)\int_{u}^bW^{(q+r)}(b-y)W^{(q)}(y-u)\diff y\diff u \notag\\
&=\int_0^bh(u)\left(W^{(q+r)}(b-u)-W^{(q)}(b-u)\right)\diff u\notag\\
&=\int_0^bh(u)W^{(q+r)}(b-u)\diff u-\rho_b^{(q)}(b;h).\notag
\end{align}
On the other hand, by Fubini's theorem together with \eqref{scale_function_laplace}
\begin{align*}
r\int_0^{\infty}e^{-\Phi(q+r)y}\rho_b^{(q)}(y;h) \diff y&=\int_0^{\infty}e^{-\Phi(q+r)y}\int_0^{b}W^{(q)}(y-u)h(u) \diff u  \diff y\\
&=r\int_0^{b}h(u)\int_{u}^{\infty}e^{-\Phi(q+r)y}W^{(q)}(y-u)\diff y \diff u\notag\\
&=\int_0^{b}e^{-\Phi(q+r)u}h(u) \diff u.\notag
\end{align*}
The above identity implies that
\begin{align}\label{g_b_5}
r\int_0^{b}e^{-\Phi(q+r)y}\rho_b^{(q)}(y;h) \diff y=\int_0^{b}e^{-\Phi(q+r)u}h(u) \diff u-r\int_{b}^{\infty}e^{-\Phi(q+r)y}\rho_b^{(q)}(y;h) \diff y.
\end{align}
Therefore, applying \eqref{g_b_2}--\eqref{g_b_5} in \eqref{delta_3}
\begin{align}\label{delta_3_a}
\delta_3&=Z^{(q+r)}(b)\dfrac{\Phi(q+r)}{r+q}\Bigg[-\dfrac{\rho^{(q)}_{b}(b;h)+g^{(q)}(b;h)}{Z^{(q)}(b)}r\int_b^{\infty}e^{-\Phi(q+r)y}Z^{(q)}(y) \diff y
\\&\quad-\int_0^{b}e^{-\Phi(q+r)u}h(u) \diff u+r\int_{b}^{\infty}e^{-\Phi(q+r)y}\rho_b^{(q)}(y;h) \diff y\Bigg]+g^{(q)}(b;h)+\int_0^bh(u)W^{(q+r)}(b-u)\diff u.\notag
\end{align}
On the other hand, using \eqref{delta_1}, \eqref{delta_2} and \eqref{delta_3_a} in \eqref{g_b_1}, 
\begin{align*}
g^{(q)}(b;h)&=Z^{(q+r)}(b)\dfrac{\Phi(q+r)}{r+q}\Bigg[-\dfrac{\rho^{(q)}_{b}(b;h)+g^{(q)}(b;h)}{Z^{(q)}(b)}r\int_b^{\infty}e^{-\Phi(q+r)y}Z^{(q)}(y)dy-\int_b^{\infty}e^{-\Phi(q+r)u}h(u)du\notag\\&+r\int_{b}^{\infty}e^{-\Phi(q+r)y}\rho_b^{(q)}(y;h)dy\Bigg]+g^{(q)}(b;h)+\dfrac{r}{r+q}Z^{(q+r)}(b)e^{-\Phi(q+r)b}g^{(q)}(b;h).
\end{align*}
Hence, solving for $g^{(q)}(b;h)$,  it yields
\begin{align}\label{r11}
C(b;h)=g^{(q)}(b;h)&=\dfrac{\Xi^{(q,r)}(b;h)-r\dfrac{\rho^{(q)}_{b}(b;h)}{Z^{(q)}(b)}\displaystyle\int_b^{\infty}e^{-\Phi(q+r) (y-b) }Z^{(q)}(y)\diff y}{\displaystyle\dfrac{r}{Z^{(q)}(b)}\int_b^{\infty}e^{-\Phi(q+r) (y-b) }Z^{(q)}(y)\diff y-\dfrac{r}{\Phi(q+r)}},
\end{align}
where $\Xi(b;h)$ is given in \eqref{e1}. Therefore, using \eqref{r11} in \eqref{r10} we get \eqref{r12}. 

\subsection{Proof of Proposition \ref{resol_par_ruin}}\label{App_Proof_2}
We denote
\begin{align}
\tilde{g}^{(q)}(x;h):=\E_x\left[\int_0^{\tau_0^{-}(r)}e^{-qt}h(U_r^{b}(t))dt\right],
\end{align}
\par (i) First, by using the Markov property and due to the absence of positive jumps, we obtain for $x<b$
\begin{align}
\tilde{g}^{(q)}(x;h)=\E_x\left[\int_0^{\tau_b^+\wedge\tau_0^-}e^{-qt}h(X_t)\diff t\right]+\E_x\left[e^{-q\tau_b^+};\tau_b^+<\tau_0^-\right] \tilde{g}^{(q)}(b;h).
\end{align}
Then, by \eqref{fpt} and \eqref{eq2} we obtain that
\begin{align}\label{r2a}
\tilde{g}^{(q)}(x;h)=\dfrac{W^{(q)}(x)}{W^{(q)}(b)}\left(\rho^{(q)}_{b}(b;h)+ \tilde{g}^{(q)}(b;h)\right)-\rho^{(q)}_{b}(x;h).
\end{align}
On the other hand, for $x > b$ we observe that $\lbrace U_r^{b}(t) : t \leq T(1)\wedge \tau_b^- \rbrace$ started at $x$ is equal in law to $\lbrace X(t) : t \leq T(1)\wedge \tau_b^- \rbrace$ started at $x$, as well. By using this, and the strong Markov property we obtain
\begin{align}\label{r7a}
\tilde{g}^{(q)}(x;h)&=\E_x\left[e^{-qT(1)};T(1)<\tau_b^-\right]\tilde{g}^{(q)}(b;h)\\
&\quad+\E_x\left[e^{-q\tau_b^-} \tilde{g}^{(q)}(X_{\tau_b^-};h);\tau_b^-<T(1)\right]+\E_x\left[\int_0^{\tau_b^-\wedge T(1)}e^{-qt}h(X_t)\diff t\right].\notag
\end{align}
Using \eqref{r2a} we have
\begin{align}\label{r3a}
\E_x\left[e^{-q\tau_b^-}\tilde{g}^{(q)}(X_{\tau_b^-};h);\tau_b^-<T(1)\right]&=\E_x\left[e^{-(q+r)\tau_b^-}\tilde{g}^{(q)}(X_{\tau_b^-};h);\tau_b^-<\infty\right]\\
&=\dfrac{(\rho^{(q)}_{b}(b;h)+\tilde{g}^{(q)}(b;h))}{W^{(q)}(b)}\E_x\left[e^{-(q+r)\tau_b^-}W^{(q)}(X_{\tau_b^-});\tau_b^-<\infty\right]\notag\\
&\quad-\E_x\left[e^{-(q+r)\tau_b^-}\rho^{(q)}_{b}(X_{\tau_{b}^{-}};h);\tau_b^-<\infty\right].\notag
\end{align}
Proceeding like in \eqref{r9} and using \eqref{lrz}, \red{it} yields 
\begin{align}\label{r9a}
\E_x&\left[e^{-q\tau_b^-}\tilde{g}^{(q)}(X_{\tau_b^-};h);\tau_b^-<T(1)\right]\\
&=\dfrac{(\rho^{(q)}_{b}(b;h)+\tilde{g}^{(q)}(b;h))}{W^{(q)}(b)}\left(W^{(q,r)}_b(x)-rW^{(q+r)}(x-b) \int_b^{\infty}  e^{-\Phi(q+r)(y-b) }W^{(q)}(y)\diff y\right)\notag\\
& \quad -\rho^{(q,r)}_b(x)+rW^{(q+r)}(x-b) \int_b^{\infty} e^{-\Phi(q+r) (y-b) }\rho^{(q)}_{b}(y;h)\diff y\notag\\
&=\dfrac{(\rho^{(q)}_{b}(b;h)+\tilde{g}^{(q)}(b;h))}{W^{(q)}(b)}\left(W^{(q,r)}_b(x)-W^{(q+r)}(x-b) Z^{(q)}(b;\Phi(q+r))\right)\notag\\
& \quad -\rho^{(q,r)}_b(x)+rW^{(q+r)}(x-b) \int_b^{\infty} e^{-\Phi(q+r) (y-b) }\rho^{(q)}_{b}(y;h)\diff y,\notag
\end{align}
where in the last equality we used \eqref{def_z_nuevo}.

Hence, combining \eqref{r1}, \eqref{r6}, and \eqref{r9a} in \eqref{r7a}, gives for $x > b$, 
\begin{align}\label{r10a}
\tilde{g}^{(q)}(x;h)&= 
\dfrac{(\rho^{(q)}_{b}(b;h)+ \tilde{g}^{(q)}(b;h))}{W^{(q)}(b)}W^{(q,r)}_b(x)-\rho^{(q,r)}_b(x)\\
&\quad-r\tilde{g}^{(q)}(b;h)\overline{W}^{(q+r)}(x-b)-\int_0^{x-b}h(y+b)e^{-\Phi(q+r)y}W^{(q+r)}(x-b-y)\diff y\notag\\
&\quad+ W^{(q+r)}(x-b)\bigg[\int_0^{\infty}h(y+b)e^{-\Phi(q+r)y}\diff y+\dfrac{r}{\Phi(q+r)}\tilde{g}^{(q)}(b;h) \notag\\
&\quad -\dfrac{(\rho^{(q)}_{b}(b;h)+ \tilde{g}^{(q)}(b;h))}{W^{(q)}(b)}r \int_0^{\infty} e^{-\Phi(q+r) z} W^{(q)}(z+b) \diff z+r\int_b^{\infty}e^{-\Phi(q+r)y (y-b) }\rho^{(q)}_{b}(y;h)\diff y\bigg].\notag
\end{align}
\par (ii) On the other hand, by an application of the Markov property we obtain that
\begin{align}\label{t_g_b_1}
\tilde{g}^{(q)}(b;h)=\tilde{\delta}_1+\tilde{g}^{(q)}(b;h)\tilde{\delta}_2+\tilde{\delta}_3,
\end{align}
where
\begin{align*}
\tilde{\delta}_1:&=\E_b\left[\int_0^{\tau_0^-\wedge T(1)}e^{-qt}h(X_t) \diff t\right],\notag\\
\tilde{\delta}_2:&=\E_b\left[e^{-qT(1)}1_{\{X_t\geq b\}};T(1)<\tau_0^-\right],\notag\\
\tilde{\delta}_3:&=\E_b\left[e^{-qT(1)}\tilde{g}^{(q)}(X_{T(1)};h)1_{\{X_{T(1)}<b\}};T(1)<\tau_0^-\right].
\end{align*}
Now, using identity \eqref{rlp} we can write
\begin{align}\label{t_delta_1}
\tilde{\delta}_1=\E_b\left[\int_0^{\tau_0^-}e^{-(q+r)t}h(X_t)dt\right]=W^{(q+r)}(b)\int_0^{\infty}e^{-\Phi(q+r)y}h(y)dy-\int_0^bW^{(q+r)}(b-y)h(y)dy,
\end{align}
and
\begin{align}\label{t_delta_2}
\tilde{\delta}_2=r\E_b\left[\int_0^{\tau_0^-}e^{-(q+r)t}1_{\{X_t\geq b\}}dt\right]=rW^{(q+r)}(b)\int_b^{\infty}e^{-\Phi(q+r)y}dy=\dfrac{r}{\Phi(r+q)}W^{(q+r)}(b)e^{-\Phi(q+r)b}.
\end{align}
In order to compute $\tilde{\delta}_3$, we  use \eqref{rlp} together with \eqref{r2a} to obtain
\begin{align}\label{t_delta_3}
\tilde{\delta}_3&=r \E_b\left[\int_0^{\tau_0^-}e^{-(q+r)t}\tilde{g}^{(q)}(X_t;h)1_{\{X_t<b\}}dt\right]\\
&=rW^{(q+r)}(b)\int_0^{b}e^{-\Phi(q+r)y}\tilde{g}^{(q)}(y;h)dy-r\int_0^bW^{(q+r)}(b-y)\tilde{g}^{(q)}(y;h)dy\notag\\
&=rW^{(q+r)}(b)\left[\dfrac{\rho^{(q)}_{b}(b;h)+\tilde{g}^{(q)}(b;h)}{W^{(q)}(b)}\int_0^{b}e^{-\Phi(q+r)y}W^{(q)}(y)dy-\int_0^{b}e^{-\Phi(q+r)y}\rho_b^{(q)}(y;h)dy\right]\notag\\
&-r\dfrac{\rho^{(q)}_{b}(b;h)+\tilde{g}^{(q)}(b;h)}{W^{(q)}(b)}\int_0^{b}W^{(q+r)}(b-y)W^{(q)}(y)dy+r\int_0^{b}W^{(q+r)}(b-y)\rho_b^{(q)}(y;h)dy.\notag
\end{align}
Using \eqref{scale_function_laplace}, gives
\begin{align}\label{t_g_b_2}
r\int_0^{b}e^{-\Phi(q+r)y}W^{(q)}(y)dy=1-r\int_b^{\infty}e^{-\Phi(q+r)y}W^{(q)}(y)dy.
\end{align}
Now, using \eqref{eq1}
\begin{align}\label{t_g_b_3}
r\int_0^{b}W^{(q+r)}(b-y)W^{(q)}(y)dy=W^{(q+r)}(b)-W^{(q)}(b).
\end{align}
Therefore, using \eqref{t_g_b_2}, \eqref{t_g_b_3}, \eqref{g_b_4}, and \eqref{g_b_5} in \eqref{t_delta_3}
\begin{align}\label{t_delta_3_a}
\tilde{\delta}_3&=W^{(q+r)}(b)\Bigg[-\dfrac{\rho^{(q)}_{b}(b;h)+\tilde{g}^{(q)}(b;h)}{W^{(q)}(b)}r\int_b^{\infty}e^{-\Phi(q+r)y}W^{(q)}(y)dy-\int_0^{b}e^{-\Phi(q+r)u}h(u)du\\
&\quad+r\int_{b}^{\infty}e^{-\Phi(q+r)y}\rho_b^{(q)}(y;h)dy\Bigg]+\tilde{g}^{(q)}(b;h)+\int_0^bh(u)W^{(q+r)}(b-u)du.\notag
\end{align}
Finally, using \eqref{t_delta_1}, \eqref{t_delta_2} and \eqref{t_delta_3_a} in \eqref{g_b_1} gives 
\begin{align*}
\tilde{g}^{(q)}(b;h)&=W^{(q+r)}(b)\Bigg[-\dfrac{\rho^{(q)}_{b}(b;h)+\tilde{g}^{{(q)}}(b;h)}{W^{(q)}(b)}r\int_b^{\infty}e^{-\Phi(q+r)y}W^{(q)}(y)dy+\int_b^{\infty}e^{-\Phi(q+r)u}h(u)du\notag\\&+r\int_{b}^{\infty}e^{-\Phi(q+r)y}\rho_b^{(q)}(y;h)dy\Bigg]+\tilde{g}^{(q)}(b;h)+\dfrac{r}{\Phi(r+q)}W^{(q+r)}(b)e^{-\Phi(q+r)b}\tilde{g}^{(q)}(b;h).
\end{align*}
Using \eqref{def_z_nuevo} and solving for $\tilde{g}^{(q)}(b;h)$, yields 
\begin{align}\label{r11a}
\tilde{C}^{(q,r)}(b;h)=\tilde{g}^{(q)}(b;h)=\dfrac{\Xi^{(q,r)}(b;h)-\dfrac{\rho^{(q)}_{b}(b;h)}{W^{(q)}(b)} Z^{(q)}(b;\Phi(q+r)) }{\displaystyle\dfrac{ Z^{(q)}(b;\Phi(q+r)) }{W^{(q)}(b)}-\dfrac{r}{\Phi(q+r)}},
\end{align}
where $\Xi(b;h)$ is given in \eqref{e1}.
Thus, by using \eqref{r11a} in \eqref{r10a}, and rearranging terms we obtain \eqref{r12aa}. 

\section{Proof of Lemma \ref{useful_iden_lemma}}\label{useful_iden}

Using itegration by parts, we obtain for $b\geq0$
\begin{align}\label{use_1}
\dfrac{r}{\Phi(r+\theta)} - \dfrac{r}{Z^{(\theta)}(b)} \int_b^{\infty} e^{-\Phi(r+\theta)(z-b)} Z^{(\theta)}(z) \diff z=-\dfrac{\theta}{\Phi(\theta+r)}\dfrac{Z^{(\theta)}(b;\Phi(\theta+r))}{Z^{(\theta)}(b)}.
\end{align}
On the other hand, recall that $\rho^{(\theta)}_{b}(x;w)=\int_0^b w(z)W^{(\theta)}(x-z)\diff z=\int_{x-b}^xw(x-u)W^{(\theta)}(u)\diff u$, then by differentiating the last expression we get
\begin{align*}
\rho^{(\theta)\prime}_{b}(x;w)=w(0)W^{(\theta)}(x)-w(b)W^{(\theta)}(x-b)+\rho^{(\theta)}_{b}(x;w'_{+}), \quad x\geq 0.
\end{align*}
Now, using integration by parts we get
\begin{align}\label{use_2}
&r\int_b^{\infty} e^{-\Phi(r+\theta)(z-b)} \rho^{(\theta)}_{b}(z;w) \diff z\\
&\quad=\dfrac{r}{\Phi(\theta+r)}\rho^{(\theta)}_{b}(b;w)+\dfrac{r}{\Phi(\theta+r)}\int_b^{\infty}e^{-\Phi(r+\theta)(z-b)} \rho^{(\theta)\prime}(z;w) \diff z\notag\\
&\quad=\dfrac{r}{\Phi(\theta+r)}\rho^{(\theta)}_{b}(b;w)+\dfrac{r}{\Phi(\theta+r)}\int_b^{\infty}e^{-\Phi(r+\theta)(z-b)} \rho^{(\theta)}_{b}(z;w'_{+}) \diff z\notag\\
&\qquad +\dfrac{1}{\Phi(\theta+r)}w(0)Z^{(\theta)}(b;\Phi(\theta+r)) -w(b)\dfrac{r}{\Phi(\theta+r)}\int_b^{\infty}e^{-\Phi(\theta+r)(z-b)}W^{(\theta)}(z-b)\diff z\notag\\
&\quad=\dfrac{r}{\Phi(\theta+r)}\rho^{(\theta)}_{b}(b;w)+\dfrac{r}{\Phi(\theta+r)}\int_b^{\infty}e^{-\Phi(r+\theta)(z-b)} \rho^{(\theta)}_{b}(z;w'_{+}) \diff z\notag\\
& \qquad+\dfrac{1}{\Phi(\theta+r)}w(0)Z^{(\theta)}(b;\Phi(\theta+r)) -\dfrac{w(b)}{\Phi(\theta+r)},\notag
\end{align}
where in the last equality we have used the fact that $\int_b^{\infty}e^{-\Phi(\theta+r)(z-b)}W^{(\theta)}(z-b)\diff z=r^{-1}$, which follows from \eqref{scale_function_laplace}. In a similar way, we obtain using integration by parts
\begin{align}\label{use_3}
\int_0^{\infty} e^{-\Phi(r+\theta)z} w(z+b)  \diff z=\dfrac{1}{\Phi(\theta+r)}w(b)+\dfrac{1}{\Phi(\theta+r)}\int_0^{\infty}e^{-\Phi(\theta+r)z}w'_{+}(z+b)\diff z.
\end{align}
Applying \eqref{use_1}, \eqref{use_2}, and \eqref{use_3} in $C^{(\theta,r)}(b;w)$ given in \eqref{e1},  gives  
\begin{align*}
C^{(\theta,r)}(b;w)+\rho^{(\theta)}_{b}(b;w)
&=\dfrac{Z^{(\theta)}(b)}{\theta}\bigg(\dfrac{\Xi^{(\theta,r)}(b,w'_{+})}{Z^{(\theta)}(b;\Phi(\theta+r))}+w(0)\bigg),
\end{align*}
where $\Xi^{(\theta,r)}(b,w'_{+})$ is as in \eqref{e4}.

The previous identity, together with \eqref{eq5}, \eqref{r11a1} and \eqref{fpt_pr}, implies that
\begin{align}\label{eq6}
&\theta W^{(\theta)}(b)\left( \dfrac{C^{(\theta,r)}(b;w) + \rho_b^{(\theta)}(b;w)}{Z^{(\theta)}(b)} \right) - w(0) W^{(\theta)}(b) - \rho^{(\theta)}_{b}(b;w'_{+}) \\
&\quad=\dfrac{W^{(\theta)}(b)}{Z^{(\theta)}(b;\Phi(\theta+r))}\bigg(\Xi^{(\theta,r)}(b;w'_{+})-\dfrac{Z^{(\theta)}(b;\Phi(\theta+r))}{W^{(\theta)}(b)}\rho^{(\theta)}_{b}(b;w'_{+})\bigg)\notag\\
&\quad=\dfrac{W^{(\theta)}(b)\tilde{C}^{(\theta,r)}(b;w'_{+})}{Z^{(\theta)}(b;\Phi(\theta+r))}\bigg(\dfrac{Z^{(\theta)}(b;\Phi(\theta+r))}{W^{(\theta)}(b)}-\dfrac{r}{\Phi(\theta+r)}\bigg)\notag\\
&\quad=\dfrac{\tilde{C}^{(\theta,r)}(b;w'_{+})}{\Phi(\theta+r)}\dfrac{Z^{(\theta)\prime}(b;\Phi(\theta+r))}{Z^{(\theta)}(b;\Phi(\theta+r))}\notag\\
&\quad=\dfrac{\tilde{C}^{(\theta,r)}(b;w'_{+})}{\Phi(\theta+r)}\dfrac{\theta W^{(\theta)}(b)}{Z^{(\theta)}(b)-\E_{b}[e^{-\theta\tau^{-}_{0}(r)}]}.\notag
\end{align}
By Proposition \ref{resol_par_ruin}, we get that $\E_{b}[\int_{0}^{\tau_{0}^{-}(r)}e^{-\theta t}w'_{+}(U_{r}^{b}(t))\diff t]=\tilde{C}^{(\theta,r)}(b;w'_{+})$. From here and \eqref{eq6},  we obtain \eqref{r13}.
\section{Proof of Some Auxiliary Results}\label{sec_proof_DPP}
\begin{lemma}\label{lem_lipschitz}
	For $y \geq x \geq 0$ and $i \in E$, we have
	\begin{equation}\label{app_1}
	0 \leq V(y,i) - V(x,i) \leq \beta(y-x).
	\end{equation}
\end{lemma}
\begin{proof}
	(i) First, we prove that $V(\cdot,i)$ is non-decreasing. Let $\pi^{(\varepsilon)} = (L^{\pi^{(\varepsilon)}}_{r}, R^{\pi^{(\varepsilon)}}_{r})$ be an $\varepsilon$-optimal strategy for $(U^{\pi^{(\varepsilon)}}_{r}(0),H(0)) = (x,i)$. For $y > x$, we define the strategy $\pi^y = (L^{\pi^y}_{r}, R^{\pi^y}_{r})$ as
	\[
	L^{\pi^y}_{r}(t) := L^{\pi^{(\varepsilon)}}_{r}(t), \quad R^{\pi^y}_{r}(t):=\left( R^{\pi^{(\varepsilon)}}_{r}(t) - y + x \right) \vee 0, \quad \text{ for } t\geq 0.
	\]
	It follows that $\pi^y$ is an admissible strategy for $(U^{\pi^y}_{r}(0),H(0)) = (y,i)$. In addition, we have
	\[
	\begin{split}
	V(y,i) - V(x,i) &\geq V_{\pi^y}(y,i) - V_{\pi^{(\varepsilon)}}(x,i) - \varepsilon \\
	& = \Em_{(x,i)} \left[\beta \int_{[0, \infty)} e^{-I(t)}  1_{\lbrace R^{\pi^{(\varepsilon)}}_{r}(t) \leq y-x \rbrace} \diff R^{\pi^{(\varepsilon)}}_{r}(t) \right] - \varepsilon.
	\end{split}
	\]
	By taking the limit as $\varepsilon \rightarrow 0$ we deduce that
	\[
	V(y,i) - V(x,i) \geq 0, \quad 0 \leq x < y.
	\]
	(ii) Now, we prove the upper bound. Let $\pi^{(y,i)} = (L^{\pi^{(y,i)}}_{r}, R^{\pi^{(y,i)}}_{r})$ be an $\varepsilon$-optimal strategy for $(U^{\pi^{(y,i)}}_{r}(0),H(0))=(y,i)$. For $0 \leq x \leq y$, we define the strategy $\pi^{(x,y,i)} = (L^{\pi^{(x,y,i)}}_{r}, R^{\pi^{(x,y,i)}}_{r})$ as
	\[
	L^{\pi^{(x,y,i)}}_{r}(t) := L^{\pi^{(y,i)}}_{r}(t) 1_{\lbrace t>0 \rbrace}, \quad R^{\pi^{(x,y,i)}}_{r}(t):= (y-x)1_{\lbrace t=0 \rbrace} + R^{\pi^{(y,i)}}_{r}(t) 1_{\lbrace t>0 \rbrace}, \quad \text{ for } t\geq 0.
	\]
	Hence,
	\[
	V(x,i) \geq V_{\pi^{(x,y,i)}}(x,i) \geq V_{\pi^{(y,i)}}(y,i) - \beta(y-x) \geq V(y,i) - \beta(y-x) - \varepsilon.
	\]
	Thus, by taking the limit as $\varepsilon \rightarrow 0$ we obtain
	\[
	V(y,i) - V(x,i) \leq \beta(y-x), \quad 0 \leq x \leq y.
	\]
\end{proof}
\begin{lemma}\label{lem_unif_bound} For all $\varepsilon > 0$ and $M > 0$ there exists a strategy $\tilde{\pi}$ such that
	\begin{equation}\label{app_2}
	\max_{i \in E} \sup_{x \in [0,M]} (V(x,i) - V_{\tilde{\pi}}(x,i)) < \varepsilon.
	\end{equation}
\end{lemma}
\begin{proof}
	Let $x_j := \dfrac{Mj}{N}$, with $j= 0,1,\cdots,N$, be a partition of $[0,M]$, where $N > M \varepsilon^{-1}$ and such that
	\begin{equation}
	\max_{i \in E} \sup_{\substack{x,y \in [0,M] \\ |x-y| < M/N}} |V(x,i) - V(y,i)| < \varepsilon.
	\end{equation}
	For all $i\in E$ and $j \in \lbrace 0,1,\cdots,N \rbrace$, let $\pi^{i,j} = (L^{i,j}_{r},R^{i,j}_{r})$ be an $\varepsilon$-optimal strategy for $(U^{\pi^{i,j}}_{r}(0),H(0)) = (x_j,i)$. For $x \in [0,M]$, we define the strategy $\tilde{\pi} = (L^{\tilde{\pi}}_{r},R^{\tilde{\pi}}_{r})$ such that $U^{\tilde{\pi}}_{r}(0) = x$ and $H(0) = i$, and for $t \geq 0$
	\[
	L^{\tilde{\pi}}_{r}(t) = L^{\pi^{i,j^*}}_{r}(t), \quad R^{\tilde{\pi}}_{r}(t) = \beta(x - x_{j^*}) 1_{\lbrace t =0 \rbrace} + R^{\pi^{i,j^*}}_{r}(t) 1_{\lbrace t> 0 \rbrace},
	\]
	where $j^*:= \min\lbrace j  : x \leq x_j \rbrace$. It follows that $|V_{\pi^{i,j^*}}(x_{j^*},i) - V_{\tilde{\pi}}(x,i)| \leq \beta(x - x_{j^*}) \leq \beta\varepsilon $.
	
	Therefore, we obtain
	\[
	\begin{split}
	| V(x,i) - V_{\tilde{\pi}}(x,i) | & \leq |V(x,i) - V(x_{j^*},i)| \\
	& \quad + |V(x_{j^*},i) - V_{\pi^{i,j^*}}(x_{j^*},i)| + |V_{\pi^{i,j^*}}(x_{j^*},i) - V_{\tilde{\pi}}(x,i)| \\
	& \leq (2 + \beta) \varepsilon \leq (2 + \beta) \varepsilon.
	\end{split}
	\]
	As the previous inequality holds for arbitrary $x \geq 0$ and $i \in E$, then \eqref{app_2} holds true.
\end{proof}
\subsection{Proof of Proposition \ref{dpp}}\label{App_DPP_Proof}
Let $x \geq 0$ and $i \in E$. Using the strong Markov property we obtain that
\[
\sup_{\pi \in \A} \Em_{(x,i)}\left[ \int_{[0, \zeta)} e^{-I(t)} \diff L^{\pi}_{ {r}}(t) - \beta \int_{[0, \zeta)}  e^{-I(t)} \diff R^{\pi}_{ {r}}(t) + e^{-I(\zeta)} V(U^{\pi}_{ {r}}(\zeta),H(\zeta)) \right] \geq V(x,i).
\]
To prove the opposite inequality, let $\pi = (L^{\pi},R^{\pi})$ be an admissible strategy and $\varepsilon > 0$. By Lemma \ref{lem_unif_bound}, for all $k \in \mathbb{Z}\cap [0,\infty)$ there exists a strategy $\pi^{k,\varepsilon}$ such that
\[
\max_{i\in E} \sup_{x \in [kM, (k+1)M]} |V(x,i) - V_{\pi^{k,\varepsilon}}(x,i)| < \varepsilon.
\]
Let us denote by $\theta$ the shift operator. Recall that $\zeta$ is the epoch of the first regime switch, then we can define a strategy $\pi^{\varepsilon} = (L^{\pi^{\varepsilon}}, R^{\pi^{\varepsilon}})$ as follows:
\begin{align*}
L^{\pi^{\varepsilon}}_{ {r}}(t) & := L^{\pi}_{ {r}}(t) 1_{\lbrace t < \zeta \rbrace} + \sum_{k=0}^{\infty} \left( L^{\pi^{k,\varepsilon}}_{ {r}}(t - \zeta) \circ \theta_{\zeta} \right) 1_{\lbrace t \geq \zeta \rbrace} 1_{\lbrace U^{\pi}_{ {r}}(\zeta) \in [k,k+1) \rbrace}, \\
R^{\pi^{\varepsilon}}_{ {r}}(t) & := R^{\pi}_{ {r}}(t) 1_{\lbrace t < \zeta \rbrace} + \sum_{k=0}^{\infty} \left( R^{\pi^{k,\varepsilon}}_{ {r}}(t - \zeta) \circ \theta_{\zeta} \right) 1_{\lbrace t \geq \zeta \rbrace} 1_{\lbrace U^{\pi}_{ {r}}(\zeta) \in [k,k+1) \rbrace}.
\end{align*}
Then we have
\begin{align*}
\Em_{(x,i)} & \left[ \int_{[0, \zeta)} e^{-I(t)} \diff L^{\pi}_{ {r}}(t) - \beta \int_{[0, \zeta)} e^{-I(t)} \diff R^{\pi}_{ {r}}(t) + e^{-I(\zeta)} V(U^{\pi}_{ {r}}(\zeta),H(\zeta)) \right] \\
& \leq \Em_{(x,i)} \left[ \int_{[0, \zeta)} e^{-I(t)} \diff L^{\pi}_{ {r}}(t) - \beta \int_{[0, \zeta)} e^{-I(t)} \diff R^{\pi}_{ {r}}(t) \right. \\
& \qquad \left. + e^{-I(\zeta)} \sum_{k=0}^{\infty} V_{\pi^{k,\varepsilon}}(U^{\pi}_{ {r}}(\zeta)H(\zeta)) 1_{\lbrace U^{\pi}_{ {r}}(\zeta) \in [k,k+1) \rbrace} + \varepsilon \right] \\
& = V_{\pi^{\varepsilon}}(x,i) + \varepsilon \leq V(x,i) + \varepsilon.
\end{align*}
As $\varepsilon > 0$ is arbitrary, the proof is complete.
\section{Proof of Lemma \ref{proof_lemma_vv}}\label{proof_lemma_v}
{Let us consider the strategy $\pi^\prime \in\A$ given by:}
\begin{align*}
{L}^{\pi^\prime}(t):=&
\begin{cases}
0, \qquad &t\in [0, T(1)) ,\\
{x\vee0}+X(T(1 ))-\inf_{s\in[0, T(1)]}X(s) , \qquad &t\in [T(1), T(2)),\\
{L}^{\pi^\prime}(T(k-1))+X(T(k))-\inf_{s\in[T(k-1), T(k)]}X(s)  , \qquad &t\in [T(k), T(k+1)), \quad k\in\N, k\geq 2. 
\end{cases}
\\
{R}^{\pi^\prime}(t):=&
\begin{cases}
-\inf_{s\in[0, t]}(X(s) - x) -x\land 0, \qquad &t\in [0, T(1)] ,\\
{R}^{\pi^\prime}(T(k))-\inf_{s\in[T(k), t]}(X(s) - X(T(k)) ), \qquad &t\in (T(k), T(k+1)],\quad k\in\N. 
\end{cases}
\end{align*}
Using the strategy $\pi'$ we now define
\begin{align*}
V_+(x, i):=&\Em_{(x,i)} \left[\int_{[0,\infty)}e^{-I(t)} \diff {L}^{\pi^\prime}(t)  \right], \\
V_-(x, i):=&\Em_{(x,i)} \left[ \int_{[0,\infty)}e^{-I(t)} \diff{L}^{\pi^\prime}(t)-\beta \int_{[0,\infty)}e^{-I(t)} \diff {R}^{\pi^\prime}(t)\right],\qquad (x, i)\in[0, \infty)\times E. 
\end{align*}
{Using the spatial homogeneity of L\'evy processes, w}e can rewrite $V_+$ and $V_-$, as follows
\begin{align*}
V_+(x, i)
=&x\Em_{(0,i)} \left[e^{-I(T(1))} \right]
+ \Em_{(0, i) }\left[\int_{[0, \infty)}e^{-I(t)}\diff {L}^{\pi^\prime}(t) \right],
 \\
V_-(x, i)=&V_+(x, i)-\beta \Em_{(0, i)} \left[\int_{[0, \infty)}e^{-I(t)}\diff {R}^{\pi^\prime}(t) \right],\qquad (x, i)\in[0, \infty)\times E. 
\end{align*}
{By Proposition \ref{gammainD} we have $V_-, V_+ \in \mathcal{D}$.} 
In addition, we note that 
\begin{align*}
\|V_- - V_+\|_{\infty} &\leq 
\beta \Em_{(0, i)} \left[\int_{[0, \infty)}e^{-I(t)}\diff {R}^{\pi^\prime}(t) \right]<\infty. 
\end{align*}
Since $\pi^\prime\in\A$, we have 
\begin{align*}
V_-(x, i)\leq V(x, i),\qquad (x,i)\in [0,\infty) \times E. 
\end{align*}
{Now, fix $\pi\in \A$}. For $(x,i)\in [0,\infty) \times E$, we have 
\begin{align}\label{4}
V_+(x, i)-V_\pi(x, i)
&=\Em_{(x, i)}\left[\int_{[0,\infty)}e^{-I(t)}\diff {L}^{\pi^\prime}(t)-\int_{[0,\infty)}e^{-I(t)}\diff L^\pi(t)+\beta \int_{[0,\infty)}e^{-I(t)}\diff R^\pi(t)\right]\\
&=\Em_{(x, i)}\left[\int_0^\infty q(H(u))e^{-I(u)}\left({L}^{\pi^\prime}(u) -L^\pi(u)+\beta  R^\pi(u)\right)  \diff u\right].\notag
\end{align}
Note that 
\begin{align}
{L}^{\pi^\prime}(0) -L^\pi(0)+  R^\pi(0)=R^\pi(0)\geq 0.\label{1}
\end{align}
For $n\in\N$, we have 
\begin{align*}
{L}^{\pi^\prime}(T(n))=&x+X(T(1))-\inf_{s\in[0, T(1)]}X(s) 
+\sum_{k=2}^n (X(T(k))-\inf_{s\in[T(k-1), T(k)]}X(s)) \\
\geq &X(T(1))+\sum_{k=2}^n (X(T(k))- X(T(k-1)))=X(T(n)). 
\end{align*}

Thus, by the definition of $R^\pi$, we have for $n\in\N$, 
\begin{align}
{L}^{\pi^\prime}(T(n)) -L^\pi(T(n))+  R^\pi(T(n))\geq X(T(n)) -L^\pi(T(n))+  R^\pi(T(n))\geq 0.
\label{2}
\end{align}
From \eqref{1} and \eqref{2}, and since $R^\pi$ is non-decreasing, we obtain for $n\in\N\cup\{0\}$ that
\begin{align}
L^{\pi^\prime}(t) -L^\pi(t)+  R^\pi(t)\geq X(T(n)) -L^\pi(T(n))+  R^\pi(t)\geq 0,\quad t\in[T(n), T(n+1)), \label{3}
\end{align}
where $T(0)=0$, {and the last inequality follows from the fact that $\pi\in\mathcal{A}$}.

{Finally, \eqref{4} and \eqref{3} imply that} $V_+(x, i)-V_\pi(x, i)\geq 0$ for $(x,i)\in [0,\infty) \times E$, hence
\begin{align*}
V(x,i)=\sup_{\pi\in\mathcal{A}}V_{\pi}(x, i)\leq V_+(x, i),\qquad (x,i)\in [0,\infty) \times E. 
\end{align*}
\section{Proof of Lemma \ref{inf_norm_finite}}\label{proof_inf_norm_finite}
From the proof of Lemma \ref{proof_lemma_vv}, we have
\begin{align*}
	V(x, i)\leq x\Em_{(0,i)} \left[e^{-I(T(1))} \right]
	+ \Em_{(0, i) }\left[\int_{[0, \infty)}e^{-I(t)}\diff {L}^{\pi^\prime}(t) \right], \quad (x,i)\in [0,\infty) \times E.
\end{align*}
On the other hand, we have 
\begin{align*}
	V_{\pi^{0, \bb}}(x, i)= &\Em_{(x,i)} \left[e^{-I(T(1))}(U^{0, \bb}_r(H(T(1)-))-\bb (H(T(1)))\lor 0) \right]
	\\
	&+\Em_{(x, i) }\left[\int_{(T(1), \infty)}e^{-I(t)}\diff {L}^{{0, \bb}}_r(t) \right]-\beta \sum_{k\in\N}\Em_{(x, i) }\left[\int_{[T(k-1), T(k))}e^{-I(t)}\diff {R}^{{0, \bb}}_r(t) \right]\\
	\geq &\Em_{(x,i)} \left[e^{-I(T(1))}(U^{0, \bb}_r(T(1)-)-\bb (H(T(1)))) \right]
	\end{align*}
	\begin{align*}
	&-\beta\sum_{k\in\N}{\left( \Em_{(0,i)} \left[e^{-I(T(1))}\right]\right)}^{k-1}\sup_{i\in E}\Em_{(0, i) }\left[\int_{[0, T(1))}e^{-I(t)}\diff (-\inf_{s\in[0, t]}X(s)) \right]\\
	=&\Em_{(x,i)} \left[e^{-I(T(1))}(X(T(1))+R^{0, \bb}_r(T(1)-)-\bb(H(T(1)))) \right]
	\\
	&-\beta\sum_{k\in\N}{\left( \Em_{(0,i)} \left[e^{-I(T(1))}\right]\right)}^{k-1}\sup_{i\in E}\Em_{(0, i) }\left[\int_{[0, T(1))}e^{-I(t)}\diff (-\inf_{s\in[0, t]}X(s)) \right]\\
	\geq &\Em_{(x,i)} \left[e^{-I(T(1))}(X(T(1))\right]-\sup_{i\in E} \bb (i)
	\\
	&-\beta\sum_{k\in\N}{\left( \Em_{(0,i)} \left[e^{-I(T(1))}\right]\right)}^{k-1}\sup_{i\in E}\Em_{(0, i) }\left[\int_{[0, T(1))}e^{-I(t)}\diff (-\inf_{s\in[0, t]}X(s)) \right]\\
	=&x\Em_{(0,i)} \left[e^{-I(T(1))}\right]+\Em_{(0,i)} \left[e^{-I(T(1))}(X(T(1))\right]-\sup_{i\in E} \bb (i)
	\\
	&-\beta\sum_{k\in\N}{\left( \Em_{(0,i)} \left[e^{-I(T(1))}\right]\right)}^{k-1}\sup_{i\in E}\Em_{(0, i) }\left[\int_{[0, T(1))}e^{-I(t)}\diff (-\inf_{s\in[0, t]}X(s)) \right]. 
\end{align*}
Thus, we have
\begin{align*}
	\| V-V_{\pi^{0,\bb}} \|_{\infty}\leq&
	\Em_{(0, i) }\left[\int_{[0, \infty)}e^{-I(t)}\diff {L}^{\pi^\prime}(t) \right]
	+\left|\Em_{(0,i)} \left[e^{-I(T(1))}(X(T(1))\right]\right|+\sup_{i\in E} \bb (i)
	\\
	&+\beta\sum_{k\in\N}{\left( \Em_{(0,i)} \left[e^{-I(T(1))}\right]\right)}^{k-1}\sup_{i\in E}\Em_{(0, i) }\left[\int_{[0, T(1))}e^{-I(t)}\diff (-\inf_{s\in[0, t]}X(s)) \right]<\infty.
\end{align*}

\end{document}